\documentclass{article}

\usepackage[margin=0.8in]{geometry}  
\usepackage{import}
\usepackage{blindtext}
\usepackage{graphicx}
\usepackage{epstopdf}
\usepackage{booktabs}
\usepackage{amsmath}
\usepackage{mathtools}
\usepackage{amsthm}
\usepackage{amssymb}
\usepackage{amsfonts}
\usepackage{cancel}
\usepackage{bm}
\usepackage{svg}
\usepackage{derivative}
\usepackage[labelsep=period]{caption} 
\usepackage{skmath}
\usepackage{comment}
\usepackage{caption}
\usepackage{subcaption}
\usepackage{xcolor} 
\usepackage[export]{adjustbox}
\usepackage{wrapfig}
\usepackage{mathtools,xparse}
\usepackage{wrapfig,lipsum,booktabs}
\usepackage{url}
\usepackage{hyperref}
\usepackage[T1]{fontenc}

\newcommand{\ue}{u_{\varepsilon}}
\newcommand{\dd}{\,\mathrm{d}}

\newtheorem{theorem}{Theorem}[section]
\newtheorem{lemma}[theorem]{Lemma}
\theoremstyle{definition}

\newtheorem{corollary}[theorem]{Corollary}

\usepackage{blindtext}

\usepackage{subfiles}

\title{\Large Modelling Capillary Rise with a Slip Boundary Condition:\\Well-posedness and Long-time Dynamics of Solutions to Washburn's Equation}
\author{Isidora Rapajić\thanks{Department of Mathematics and Informatics,
Faculty of Sciences, Trg Dositeja Obradovića 4, 21000 Novi Sad, Serbia, and Mathematical Institute of the Serbian Academy of Sciences and Arts, Belgrade, Serbia. Email: \texttt{isidora.rapajic@turing.mi.sanu.ac.rs} },~ Srboljub Simić\thanks{Department of Mathematics and Informatics,
Faculty of Sciences, Trg Dositeja Obradovića 4, 21000 Novi Sad, Serbia. Email: \texttt{ssimic@uns.ac.rs}},~ and Endre S{\"u}li\thanks{Mathematical Institute, University of Oxford, Andrew Wiles Building, Woodstock Road, Oxford OX2 6GG, United Kingdom. Email: \texttt{endre.suli@maths.ox.ac.uk}}}
\date{}

\begin{document}

\maketitle

\begin{abstract} 
The aim of this paper is to extend Washburn's capillary rise equation by incorporating a slip condition at the pipe wall. The governing equation is derived using fundamental principles from continuum mechanics. A new scaling is introduced, allowing for a systematic analysis of different flow regimes. We prove the global-in-time existence and uniqueness of a bounded positive solution to Washburn's equation that includes the slip parameter, as well as the continuous dependence of the solution in the maximum norm on the initial data. Thus, the initial-value problem for Washburn's equation is shown to be well-posed in the sense of Hadamard. Additionally, we show that the unique equilibrium solution may be reached either monotonically or in an oscillatory fashion, similarly to the no-slip case. Finally, we determine the basin of attraction for the system, ensuring that the equilibrium state will be reached from the initial data we impose. These results hold for any positive value of the nondimensional slip parameter in the model, and for all values of the ratio $h_0/h_e$ in the range $[0,3/2]$, where $h_0$ is the initial height of the fluid column and $h_e$ is its equilibrium height.
\end{abstract}


\section{Introduction} 


The equation governing the spontaneous movement of a liquid in a vertical narrow pipe, derived and analyzed by E.~W.~Washburn in 1921 \cite{washburn1921dynamics}, is known as Washburn's equation. The terms appearing in the equation are inertial, viscous, gravitational, and capillary.
Although Washburn's equation can be derived using Newton’s second law, it is also possible to derive it from first principles, that is, from the laws of mass and momentum conservation. 

The case of liquid-rise between two parallel plates is discussed in \cite{grunding2020enhanced}, and a similar procedure for a cylindrical geometry is considered in \cite{kornev2001spontaneous}.
It has been observed that the stationary state can be reached monotonically or in an oscillatory manner \cite{quere1997inertial, quere1999rebounds, kornev2001spontaneous}. However, discrepancies between experimental observations and predictions based on the model were observed during the initial phase of the flow. One possible explanation could be the Huh--Scriven paradox. Specifically, if the velocity field is assumed to be parabolic with a no-slip boundary condition, this would imply that the velocity at the pipe wall is zero, even as the liquid column rises \cite{huh1971hydrodynamic}. This led Fricke et al.~and Gr\"{u}nding \cite{fricke2023analytical, grunding2020enhanced} to the introduction of a slip condition in the velocity field to account for the viscous dissipation caused by the moving contact line.

The question of global existence and uniqueness of solutions to Washburn's equation over the time interval $t \in [0,\infty)$ was considered in \cite{plociniczak2018monotonicity}, which provided both the motivation and inspiration for this work. However, closer examination revealed gaps in the proof that require further attention. A crucial hypothesis of the Picard--Lindel\"{o}f theorem (namely, the Lipschitz condition) does not hold at $t=0$ because of the choice of  zero initial data for the solution and the fact that, while the nonlinear term $f(u):=1 - \sqrt{2u}$ featuring in the equation satisfies a H\"{o}lder condition with exponent $1/2$ on the space of nonnegative continuous functions defined on the interval $[0,\infty)$ that vanish at $t=0$, it violates the Lipschitz condition. The authors of \cite{plociniczak2018monotonicity} therefore applied
Schauder's fixed point theorem to prove the existence of a solution defined on $[0,\infty)$, in conjunction with Theorem 11.7 from \cite{precup2002methods}, to guarantee the local uniqueness of the solution over the interval $[0,1/2]$, followed by the use of Banach's fixed point theorem over the interval $[1/2,\infty)$ to extend the locally unique solution to a globally unique solution over the whole of $[0,\infty)$. While the paper \cite{plociniczak2018monotonicity} contains a number of interesting ideas, the proof of existence and uniqueness of solutions to Washburn's equation presented there is incomplete, as in their application of both Schauder's and Banach's fixed point theorem the authors failed to verify the essential requirement that their proposed fixed point map $T: v \mapsto T(v)$ is a self-map, i.e., that $T$ maps a certain complete metric space $\mathcal{X}$  (which should have been taken to be the space of all bounded nonnegative continuous functions $v$ defined on the interval $[0,\infty)$ in the case of Schauder fixed point theorem, and the space of all continuous functions $v$ defined on a bounded closed interval which are equi-bounded from below by a positive constant in the case of Banach fixed point theorem) into itself. We were unable to fill the gaps in the authors' proof of Theorem 1 in \cite{plociniczak2018monotonicity}. Furthermore, the application of Theorem 11.7 from
\cite{precup2002methods} in the authors' proof of Theorem 1 in \cite{plociniczak2018monotonicity} is flawed because, with $\varphi(s) \equiv 0$ and $\psi(s)=s^2/2$ chosen by the authors as the `endpoints’ of the interval of monotonicity in the proof of their Theorem 1, Theorem 11.7 from \cite{precup2002methods} cannot be applied since  it requires strict positivity of the two endpoints of the interval of monotonicity over the range of $s$ under consideration, which in the paper \cite{plociniczak2018monotonicity} is the interval $s \in [0,1/2]$.
For these reasons, we shall revisit the proof of global existence and uniqueness of solutions to Washburn's equation through a completely different approach, and we shall do so by admitting a broader range of initial data than in \cite{plociniczak2018monotonicity}.

Thus, our aim is twofold. In terms of physics, our objective is to place Washburn's equation into the broadest possible context and derive it using a precise list of assumptions. This assumes the application of the basic laws of continuum mechanics, which will provide a firm basis for generalizations. In terms of mathematics, our aim is to apply rigorous tools in the analysis of model reduction, to prove the existence and uniqueness
of solutions to Washburn's equation over the entire time interval $[0,\infty)$, and to explore the stability of solutions to Washburn's equation. This, in turn, will bring about simplifications and extensions of known results. 

The paper is organized as follows. In Section \ref{sec::Derivation}, we derive Washburn’s equation for a pipe with constant radius, introducing a slip condition analogous to the approach in \cite{fricke2023analytical}. We shall first use the local forms of mass and momentum balance in cylindrical coordinates, along with basic assumptions on the fluid and its flow, to derive restrictions on the velocity and pressure fields. Additional key ingredients, which will ultimately lead to Washburn's equation, are our assumption of the velocity profile, the introduction of a mean velocity, and the use of momentum balance in its global form. 

In Section \ref{sec::DimensionalAnalysis} we focus on the dimensional analysis of the model and its transformation to a dimensionless form. This will be the starting point for the subsequent mathematical analysis. Furthermore, using a scaling argument we develop a rigorous approach to model reduction in the context of Washburn's equation. It will reveal the conditions which allow us to neglect particular physical mechanisms. 

In Section \ref{sec::GlobalExistence} we prove the global-in-time existence and uniqueness of solutions to Washburn's equation. Our main tool is the Picard--Lindel\"{o}f theorem applied to a regularized form of Washburn's equations, in conjunction with a compactness argument based on the Arzel\`{a}--Ascoli theorem. To the best of our knowledge, this is the first instance of the application of these techniques to Washburn's equation. In this study we admit more general (but physically meaningful) initial data than the homogeneous initial data used in \cite{plociniczak2018monotonicity}. We also prove the continuous dependence of the solution on the initial data in the maximum norm. 

Finally, in Section \ref{sec::Stability} we discuss the stability of the unique stationary (equilibrium) solution, corresponding to the equilibrium height, and the convergence of the solution of the initial-value problem to the equilibrium as $t \to + \infty$. Linear stability analysis is used to confirm  the asymptotic stability of the stationary solution; it reveals the influence of the slip parameter on the critical value of the model parameter $\omega$ that separates monotonic solutions from oscillatory solutions. This local characterization based on linearization raises the question of nonlinear stability of the equilibrium. An appropriate choice of a Lyapunov function and the application of LaSalle's invariance principle enable us to identify the basin of attraction, which includes the initial state of the system proposed in this study. This global stability result, in conjunction with global existence and uniqueness, implies that solutions to the initial-value problem studied in this paper converge to the equilibrium state. 

We conclude with some remarks and an outlook on potential future research directions. 



\section{Derivation of the model} 
\label{sec::Derivation}


The usual approach to the modeling of capillary flow of a fluid is to apply Newton's second law to a control volume (or an infinitesimal volume element). This simplified procedure yields a second-order ordinary differential equation, Washburn's equation, which encodes all physical mechanisms relevant for the flow: inertia, gravity, viscosity and capillary effects. In our view, the main drawback of this approach is that it is difficult to generalize to more complex situations that are physically plausible. Therefore, our goal in this section is to derive Washburn's equation starting from first principles, i.e., from the fundamental laws of mass and momentum balance for continuous media. 

\subsection{Modelling assumptions} 

We study the motion of a fluid through a vertical cylindrical pipe of circular cross section. It is assumed that the cross-sectional radius $R$ of the pipe is constant. It is also assumed that the fluid is Newtonian, of constant mass density $\rho$ and of constant dynamic viscosity $\mu$. During the motion, the free surface of the fluid is assumed to be subject to constant surface tension $\gamma$ and a constant contact angle $\theta$. The volume force that acts on the bulk fluid is gravity of constant acceleration $g$. The main assumption concerning the motion of the fluid is that the velocity field $\mathbf{v}$ is axially symmetric and parallel to the axis of the pipe. 

Our goal is to determine the height $h(t)$ of the meniscus at time $t \geq 0$. The influence that the approximately spherical cap of the free surface has on the motion of the bulk fluid will be neglected. It will also be assumed that the entrance into the pipe is
not immersed in the large (theoretically infinite) reservoir, in which the fluid is still. These assumptions make the model approximate from the outset, but it will turn out that they are needed for the derivation of Washburn's equation. 

\subsection{Local form of the balance laws} 

The first step in the derivation of the model will be to use the mass and momentum balance laws in their local form \cite{acheson1990elementary}. Since the fluid is assumed to be incompressible and viscous, these are the Navier--Stokes equations for a viscous incompressible Newtonian fluid: 
\begin{equation}
	\nabla \cdot \mathbf{v} = 0, 
	\label{eq::MassBL}
\end{equation}
\begin{equation}
	\rho \left( \frac{\partial \mathbf{v}}{\partial t} + \mathbf{v} \cdot \nabla \mathbf{v} \right) = - \nabla p + \mu\, \Delta \mathbf{v} + \rho\, \mathbf{g}. 
	\label{eq::MomentumBL}
\end{equation}
In component-wise form, the mass and momentum balance laws in cylindrical coordinates are given in the Appendix, equations \eqref{eq::MassBalanceCylindrical} and \eqref{eq::MomentumBalanceCylindrical}, respectively. They will be used to derive further restrictions on the velocity and pressure fields. 
\begin{enumerate}
	\item \textbf{The velocity field} in cylindrical coordinates with respect to the usual orthonormal basis is expressed as
	\begin{equation*}
		\mathbf{v} = v_r \mathbf{e}_r + v_{\varphi} \mathbf{e}_{\varphi} + v_z \mathbf{e}_z.
	\end{equation*}
	Since the fluid flows along the (vertical) $z$-axis, we assume that
	\begin{equation}
		v_r \equiv 0, \quad v_{\varphi} \equiv 0. 
		\label{eq::VelocityAssumptions}
	\end{equation}
	It follows from the mass balance law \eqref{eq::MassBalanceCylindrical} and the assumptions on the velocity \eqref{eq::VelocityAssumptions} that
	\begin{equation}
		\frac{\partial v_z}{\partial z} = 0.
		\label{eq::VelocityField}
	\end{equation}
	In other words, $v_z$ does not depend on $z$. Because of the assumed axial symmetry, we may assume that $v_z$ does not depend on $\varphi$ either, so $v_z = v_z(r,t)$.
	It then follows that the velocity field has the form 
	\begin{equation}
		\mathbf{v} = v_z(r,t)\mathbf{e}_z. 
		\label{eq::VelocityField-Local}
	\end{equation}
	\item \textbf{The pressure field} in our model is assumed to be stationary, i.e., it is supposed to be independent of time $t$.  From the momentum balance law \eqref{eq::MomentumBalanceCylindrical} and the assumptions \eqref{eq::VelocityAssumptions}  on the velocity field, we obtain the following constraints on the pressure field:
	\begin{equation*}
		\frac{\partial p}{\partial r} = 0, \quad \frac{\partial p}{\partial \varphi} = 0.
	\end{equation*}
	In other words, the pressure field depends only on $z$, but not on $r$ and $\varphi$; that is,
	\begin{equation}
		p = p(z).
		\label{eq::PressureField}
	\end{equation} 
\end{enumerate}

The results of the local analysis, expressed in the equalities \eqref{eq::VelocityField-Local} and \eqref{eq::PressureField}, considerably simplify the Navier--Stokes equations \eqref{eq::MassBL}, \eqref{eq::MomentumBL}, but we are still faced with a partial differential equation that has to be solved. To arrive at Washburn's equation, which is an ordinary differential equation, we need to impose further assumptions. Also, we need to analyze the motion using global balance laws, rather than local ones. 

\subsection{Flow field and mean velocity} 

Since the use of a parabolic velocity profile is a standard assumption for the motion of Newtonian fluids through cylindrical domains \cite{zhmud2000dynamics}, we shall assume that the velocity field of the fluid is described by a Poiseuille flow profile (see Figure \ref{fig::Pipe}). In addition, we introduce a slip parameter, as in \cite{fricke2023analytical}, in order to model the moving contact line:
\begin{equation}
	v_{z}(r,t) = v(t) \left( 1 - \frac{r^2}{R^2} + \frac{2L}{R}  \right),
	\label{eq::Poiseuille}
\end{equation}
where a slip-length $L \geq 0$ is introduced; in the no-slip case $L=0$. 

\begin{figure}[ht]
	\centering
	\includegraphics[width=0.25\textwidth]{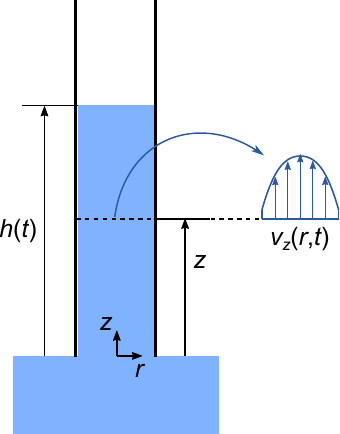}
	\caption{Capillary pipe with Poiseuille velocity profile and column height $h(t)$ at time $t$.}
	\label{fig::Pipe}
\end{figure}

Let us note that postulating a no-slip velocity profile, although commonly done for Newtonian fluids, in this case becomes paradoxical: the fluid velocity at the pipe wall, $r=R$, is then zero at all times, so the contact line would remain stationary; at the same time the liquid column is assumed to be rising, which implies vertical motion of the contact line. This is known as the Huh--Scriven paradox \cite{huh1971hydrodynamic}. 

Since the velocity field is uniform along the vertical axis $z$ of the pipe, we may introduce the mean velocity $\bar{v}$ of the circular cross section, perpendicular to the $z$-axis. It is defined as the mean value of $v_{z}(r,t)$ over the cross section (see the Appendix): 
\begin{equation}
	\bar{v}(t) := \frac{1}{R^{2} \pi} \int_{0}^{2\pi} \int_{0}^{R} v_{z}(r,t) r \dd r \dd \varphi = \frac{v(t)}{2 \beta},
	\label{eq::MeanVelocity}
\end{equation}
where the slip parameter $\beta$ is defined as
\begin{equation}
	\beta := \left(1 + 4\frac{L}{R}\right)^{-1};
	\label{eq::slipParameter}
\end{equation} 
in the no-slip case $\beta = 1$. In our analysis, the mean velocity \eqref{eq::MeanVelocity} will be taken to be related to the rate of change of the height $h(t)$ of the fluid column at time $t$ by
\begin{equation}
	\bar{v}(t) = \frac{\dd h(t)}{\dd t}. 
	\label{eq::v-hRelation}
\end{equation}

\subsection{Global form of the momentum balance law}

To derive Washburn's equation with a slip boundary condition, we apply the momentum balance law in its global/integral form \cite{Gurtin_Fried_Anand_2010}, taking into account all  of the relevant assumptions stated above. The rate of change of the linear momentum is balanced by the sum of the contact forces and body forces, i.e., 
\begin{equation*}
	\frac{\dd }{\dd t} \int_{\mathcal{P}_{t}} \mathbf{K}(\mathbf{x},t) \dd V = \int_{\partial \mathcal{P}_{t}} \mathbf{t}(\mathbf{x},t,\mathbf{n}) \dd S + \int_{\mathcal{P}_{t}} \rho \mathbf{b}(\mathbf{x},t) \dd V, 
\end{equation*}
where $\mathbf{K} = \rho \mathbf{v}$ is the momentum density, $\mathbf{t}$ is the traction at the boundary point with outward unit normal vector $\mathbf{n}$, and $\mathbf{b}$ is the specific volume force. The domain of integration $\mathcal{P}_{t}$ is the volume of the pipe filled with fluid at time $t$. Using the previously introduced assumptions, the following equalities are obtained: 
\begin{equation}
	\begin{split}
		\frac{\dd }{\dd t} \int_{\mathcal{P}_{t}} \mathbf{K}(\mathbf{x},t) \dd V & = \rho R^{2} \pi \frac{\dd }{\dd t} \left[ h(t) \bar{v}(t) \right] \mathbf{e}_{z}, 
		\\ 
		\int_{\partial \mathcal{P}_{t}} \mathbf{t}(\mathbf{x},t,\mathbf{n}) \dd S & = - R^{2} \pi [p(h(t)) - p(0)] \mathbf{e}_{z} - 8 \mu \pi \beta h(t) \bar{v}(t) \mathbf{e}_{z}, 
		\\ 
		\int_{\mathcal{P}_{t}} \rho \mathbf{b}(\mathbf{x},t) \dd V & = - \rho g R^{2} \pi h(t) \mathbf{e}_{z}. 
	\end{split}
	\label{eq::WashburnPrep}
\end{equation}

Taking into account the relation \eqref{eq::v-hRelation} and the fact that the pressure drop is balanced by the vertical component of the surface tension we have that 
\begin{equation*}
	p(0) - p(h(t)) = \frac{2 \gamma \cos \theta}{R}.
\end{equation*}
Combining these equalities, the following nonlinear second-order ordinary differential equation is obtained, which governs the evolution of the height $h(t)$ of the fluid column/meniscus:
\begin{equation}
	\frac{\dd }{\dd t} \left[ \rho h(t) \frac{\dd h(t)}{\dd t} \right] + \frac{8 \mu \beta}{R^{2}} h(t) \frac{\dd h(t)}{\dd t}
	+ \rho g h(t) = \frac{2 \gamma \operatorname{cos}\theta}{R}. 
	\label{eq::dimensionalWashburn}
\end{equation}
We shall refer to this equation as Washburn's equation. The additive terms on the left-hand side are the inertial, viscous and gravity term, respectively, whereas on the right-hand side there is a single term, the capillary term. 

Two remarks are in order concerning the model that will be of use in the subsequent analysis. First, the slip at the boundary affects the model only through the slip parameter $\beta$, which appears in the viscous term. Second, Washburn's equation \eqref{eq::dimensionalWashburn} has a unique stationary solution $h(t) = h_{e} = \mathrm{const.}$, where
\begin{equation}
	h_{e} := \frac{2 \gamma \operatorname{cos}\theta}{\rho g R}, 
	\label{eq::Jurin_height}
\end{equation}
which is known as Jurin's (equilibrium) height; clearly, the value of the slip parameter $\beta$ does not affect the equilibrium height \eqref{eq::Jurin_height}. 

Our final observation is concerned with the choice of the initial data. As is evident from \eqref{eq::dimensionalWashburn}, Washburn's equation
becomes singular at $t=0$ if homogeneous initial conditions $h(0) = 0$, $\dd h(0)/\dd t = 0$ are imposed. If this singularity is to be avoided
in \eqref{eq::dimensionalWashburn}, such homogeneous initial conditions cannot be applied. There have been different attempts to introduce appropriate initial conditions that are physically admissible and mathematically correct at the same time. In this study it will be assumed that the fluid column is initially at rest, with the initial height $h_{0}$ assumed to be nonnegative, but much smaller than the equilibrium height $h_{e}$: 
\begin{equation}
	h(0) = h_{0} \ll h_{e}, \quad 
	\frac{\dd h(0)}{\dd t} = 0.
	\label{eq::InitCond_Dimensional}
\end{equation}

In view of our assumptions, Washburn's equation \eqref{eq::dimensionalWashburn} with initial data \eqref{eq::InitCond_Dimensional} may be seen to be a  suitable model of capillary rise on the entire flow domain, except perhaps in a small neighborhood of the entrance $z = 0$ and in a small neighborhood of the meniscus $z = h(t)$. 



\section{Dimensional analysis and model reduction} 
\label{sec::DimensionalAnalysis}


The systematic mathematical analysis of Washburn's equation, to be conducted in the sequel, requires a model in dimensionless form. In this study we shall adopt the same dimensionless variables as the ones used in \cite{plociniczak2018monotonicity}: 
\begin{equation}\label{eq::dimensionless}
	H := \frac{h}{h_{e}}, \quad T := \frac{t}{\tau}, \quad 
	\tau := \frac{8 \mu h_{e}}{\rho g R^2}, \quad
	\alpha := \frac{h_{0}}{h_{e}}.
\end{equation}
These lead to the following dimensionless form of Washburn's equation:
\begin{equation}
	\omega(HH')' + \beta HH' + H = 1.
	\label{switala_scaled_with_imemersion1}
\end{equation}
with the dimensionless parameter $\omega$ defined by 
\begin{equation}
	\omega := \frac{h_e}{g \tau^{2}} = \frac{\rho ^{2} R^{4} g}{64 \mu^{2} h_e}, 
	\label{omega_asymp}
\end{equation}
where $H = H(T)$ and prime denotes differentiation with respect to the dimensionless time $T$. The initial conditions are scaled as follows:
\begin{equation}
	H(0) = \alpha, \quad H'(0) = 0. 
	\label{eq::IV_DLess}
\end{equation}
The stationary (equilibrium) solution $H(T) = H_{e} = \mathrm{const.}$ of \eqref{switala_scaled_with_imemersion1}, which corresponds to \eqref{eq::Jurin_height}, has the form 
\begin{equation}
	H_{e} = 1. 
	\label{eq::Jurin_DLess}
\end{equation}

Although the physical origins of the terms appearing in the equation \eqref{switala_scaled_with_imemersion1} are now hidden, they can be recognized as follows: $\omega(HH')'$ is the inertial term,  $\beta HH'$ is the viscous term, $H$ is the gravity term, and $1$ is the capillary term. Further analysis will also be facilitated if the dimensionless parameter $\omega$ is expressed in terms of dimensionless numbers, the Ohnesorge number $\mathrm{Oh}$ and the Bond number $\mathrm{Bo}$: 
\begin{align}\label{eq::omega_Bo-Oh}
	\begin{aligned}
		\omega = \frac{1}{128 \operatorname{cos}\theta} \left( \frac{\mathrm{Bo}}{\mathrm{Oh}} \right)^{2},    
		\\ 
		\mathrm{Oh} := \frac{\mu}{\sqrt{R\rho\gamma}}, \quad \mathrm{Bo} := \frac{\rho g R^2}{\gamma}. 
	\end{aligned}
\end{align}
The Ohnesorge number measures the relative importance of viscosity compared to inertia and surface tension, whereas the Bond number (also known as the E\"{o}tv\"{o}s number) relates gravitational forces to surface tension.  

Model reduction is a technique in which certain terms in the model are canceled out because of negligible influence of the corresponding physical mechanism. Model reduction is usually based on heuristic arguments and is performed in an \emph{ad hoc} manner. Our aim is to analyze possible model reductions in a systematic way. To that end we need a properly scaled governing equation that can be simplified if a certain term is much smaller than the others, which are in turn of the same order of magnitude. The dimensionless Washburn's equation \eqref{switala_scaled_with_imemersion1} is a good starting point, but it cannot capture all the possibilities that are likely to appear. The reason for this is simple: it represents only one possible scaled governing equation. For a systematic analysis of model reduction we have to construct a whole family of scaled equations.

Let us introduce the new scalings
\begin{equation*}
	t^{\ast} := \frac{T}{\hat{t}}, \quad 
	h^{\ast} := \frac{H}{\hat{h}},
\end{equation*}
and plug them in the equation \eqref{switala_scaled_with_imemersion1} to get
\begin{equation}
	\omega \frac{\hat{h}^{2}}{\hat{t}^{2}} \frac{\mathrm{d}}{\mathrm{d}t^{\ast}} \left( h^{*} \frac{\mathrm{d}h^{\ast}}{\mathrm{d}t^{\ast}}\right) + \beta \frac{\hat{h}^{2}}{\hat{t}}h^{\ast} \frac{\mathrm{d}h^{\ast}}{\mathrm{d}t^{\ast}} + \hat{h}h^{\ast} = 1.
	\label{newscaled}
\end{equation}
We shall assume that the scaled time $\hat{t}$ and height $\hat{h}$ are power law functions of $\omega$, that is:
\begin{equation*}
	\hat{t} = \omega^{a}, \quad \hat{h}= \omega^{b}.
\end{equation*}
With these assumptions, equation \eqref{newscaled} becomes
\begin{equation}
	\omega^{2b + 1 - 2a} (h^{\ast}{h^{\ast}}')' + \omega^{2b - a} \beta h^{\ast}{h^{\ast}}' + \omega^{b}h^{\ast} = 1.
	\label{newscaled_omega}
\end{equation}
Note that the capillary term on the right-hand side is of order of unity and cannot be further scaled. Therefore, the other terms in \eqref{newscaled_omega} have to be balanced with it. There are three terms on the left-hand side of different physical origins (inertia, viscosity, and gravity) and it turns out that there are four possibilities for reduction. We shall consider these four cases separately. 

\textbf{Case 1: Negligible gravity.} In this case, capillary forces are balanced by inertial and viscous terms, which are supposed to be of the order of unity. In other words, the coefficients of inertial and viscous term have to be equal to $1$, that is
\begin{equation*}
	\begin{split}
		2b+ 1 - 2a & = 0, \\
		2b - a & = 0,
	\end{split}
\end{equation*}
which implies that
\begin{equation*}
	a = 1, \quad b = \frac{1}{2}.
\end{equation*} 
Therefore, equation \eqref{newscaled_omega} becomes
\begin{equation*}
	(h^{*}{h^{*}}')' + \omega^{1/2}h^{*} + \beta h^{*}{h^{*}}'= 1.
\end{equation*}
Letting $\omega \to 0$ we obtain the following simplified equation for the case of a negligible gravity term:
\begin{equation}
	(h^{*}{h^{*}}')' + \beta h^{*}{h^{*}}' = 1.
	\label{newscaled_gravity}
\end{equation}

\textbf{Case 2: Negligible inertia.} In this case, capillary forces are balanced by the viscous and gravitational terms. Similarly as in the case above, if the coefficient of 
the inertial term has to be small, the viscous and gravitational terms should be of the order of unity. It follows that 
\begin{equation*}
	\begin{split}
		2b - a & = 0, \\
		b  & = 0,
	\end{split}
\end{equation*}
which implies that
\begin{equation*}
	a = b = 0.
\end{equation*} 
This then transforms equation \eqref{newscaled_omega} into
\begin{equation*}
	\omega (h^{*}{h^{*}}')'  + h^{*} + \beta h^{*}{h^{*}}' = 1.
\end{equation*}
Letting $\omega \to 0$ we obtain the following simplified equation for the case of a negligible inertial term:
\begin{equation}
	\beta h^{*}{h^{*}}' + h^{*} = 1.
	\label{newscaled_inertia}
\end{equation}

\textbf{Case 3: Negligible gravity and inertia.}
Since in both Case 1 and Case 2 the model was reduced for $\omega \to 0$, it is natural to raise the question whether it is possible to neglect both gravity and inertia. In that case, the capillary term should be balanced by the viscous term only, which implies that its coefficient must be of the order of unity, i.e., 
\begin{equation*}
	2 b - a = 0 \quad \Rightarrow \quad a = 2b. 
\end{equation*}
At the same time, the powers of $\omega$ in the other two coefficients have to be positive, that is, 
\begin{equation*}
	\begin{split}
		2 b + 1 - 2 a & > 0, \\
		b & > 0,
	\end{split}
\end{equation*}
which together imply the following one-parameter family of solutions: 
\begin{equation*}
	a \in (0,1), \quad 
	b \in \left( 0, \frac{1}{2} \right), \quad
	a = 2b. 
\end{equation*}
Consequently, equation \eqref{newscaled_omega} becomes
\begin{equation*}
	\omega^{1 - a} (h^{\ast}{h^{\ast}}')' + \beta h^{\ast}{h^{\ast}}' + \omega^{a/2}h^{\ast} = 1.
\end{equation*}
Letting $\omega \to 0$ we obtain the following reduced model with neglected gravity and inertia:
\begin{equation}
	\beta h^{\ast}{h^{\ast}}' = 1.
	\label{newscaled_gravity+inertia}
\end{equation}
This equation is of particular significance since in the no-slip case ($\beta = 1$) its solution is  
\begin{equation*}
	h^{\ast}(t^{\ast}) \propto \sqrt{t^{\ast}}, 
\end{equation*}
which is known as Washburn solution (in a restricted sense). 

\textbf{Case 4: Negligible viscosity.} In this case, capillary forces are balanced by inertial and gravitational terms. If the inertial and gravitational terms are of the order of unity, it follows that 
\begin{equation*}
	\begin{split}
		2b + 1 - 2a & = 0, \\
		b & = 0,
	\end{split}
\end{equation*}
which implies that
\begin{equation*}
	a = \frac{1}{2}, \quad b = 0,
\end{equation*}
and reduces equation \eqref{newscaled_omega} to 
\begin{equation*}
	(h^{*}{h^{*}}')' + h^{*} + \omega^{-1/2} \beta h^{*}{h^{*}}' = 1.
\end{equation*}
It is obvious that the viscous term cannot become negligible when $\omega \to 0$. However, it may be neglected for $\omega \to \infty$, leading to the following simplified equation:
\begin{equation}
	(h^{*}{h^{*}}')' + h^{*} = 1.
	\label{newscaled_viscosity}
\end{equation}

The reduction of the model presented above was based on the size of the dimensionless parameter $\omega$. In \eqref{eq::omega_Bo-Oh} $\omega$ was expressed in terms of two dimensionless numbers, $\mathrm{Oh}$ and $\mathrm{Bo}$, which will help reveal the physical circumstances when model reduction is feasible. Equations \eqref{newscaled_gravity}, \eqref{newscaled_inertia} and \eqref{newscaled_gravity+inertia} were derived for $\omega \to 0$. In Case 1 (negligible gravity), this is likely to occur when $\mathrm{Bo} \to 0$ (strong surface tension with respect to gravity, the liquid column is rising). In Case 2 (negligible inertia) this is expected to happen because of $\mathrm{Oh} \to \infty$ (strong viscous forces compared to inertia). In Case 3 (negligible gravity and inertia), both effects can be present, i.e., $\mathrm{Bo} \to 0$ and $\mathrm{Oh} \to \infty$. On the other hand, equation \eqref{newscaled_viscosity}, in which the viscosity term is neglected (Case 4), is obtained for $\omega \to \infty$. This is possible for $\mathrm{Oh} \to 0$, when the viscosity is small compared to inertia and surface tension. 
The other possibility would be $\mathrm{Bo} \to \infty$, which means that the gravitational effect should be dominant with respect to the surface tension. However, this is not the case since we chose the scaling in such a way that the capillary term (surface tension) is of the order of unity. 

\begin{figure}[ht] 
	\centering
	\begin{subfigure}{0.42\textwidth}
		\includegraphics[width=\linewidth]{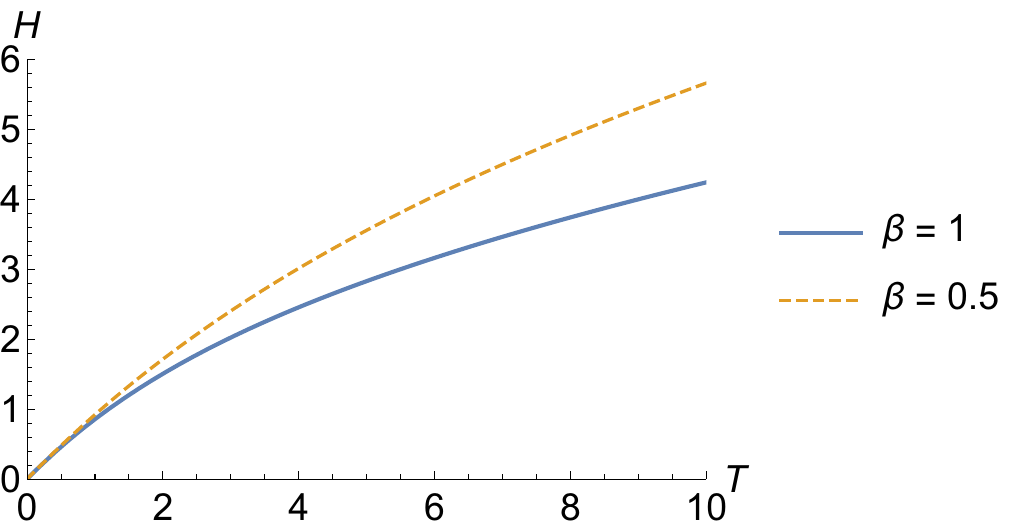}
		\caption{Case 1: \\$a = 1$, $b = 1/2$, $\omega \to 0$}
		\label{fig:subim1}
	\end{subfigure}
	\hspace{1cm}
	\begin{subfigure}{0.42\textwidth}
		\includegraphics[width=\linewidth]{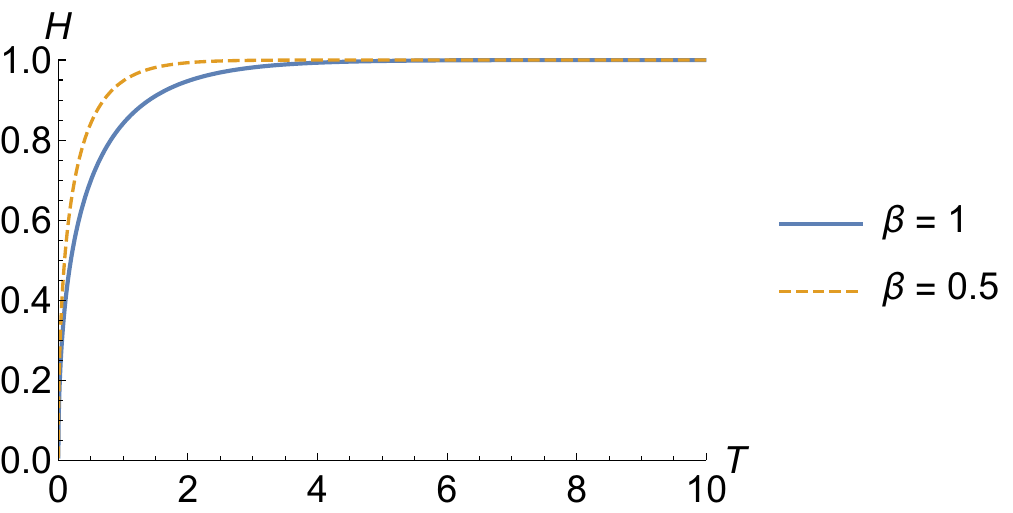}
		\caption{Case 2:\\ $a = 0$, $b = 0$, $\omega \to 0$}
		\label{fig:subim2}
	\end{subfigure}
	\\~\\
	\begin{subfigure}{0.42\textwidth}
		\includegraphics[width=\linewidth]{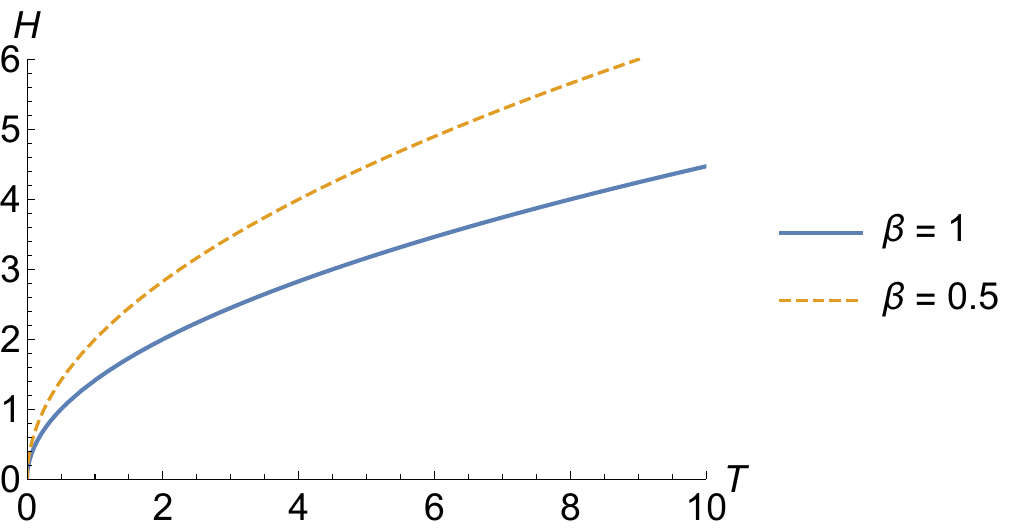}
		\caption{Case 3: \\  $a \in (0,1)$, $b \in (0, 1/2)$,  $\omega \to 0$}
		\label{fig:subim3}
	\end{subfigure}
	\hspace{1cm}
	\begin{subfigure}{0.42\textwidth}
		\includegraphics[width=\linewidth]{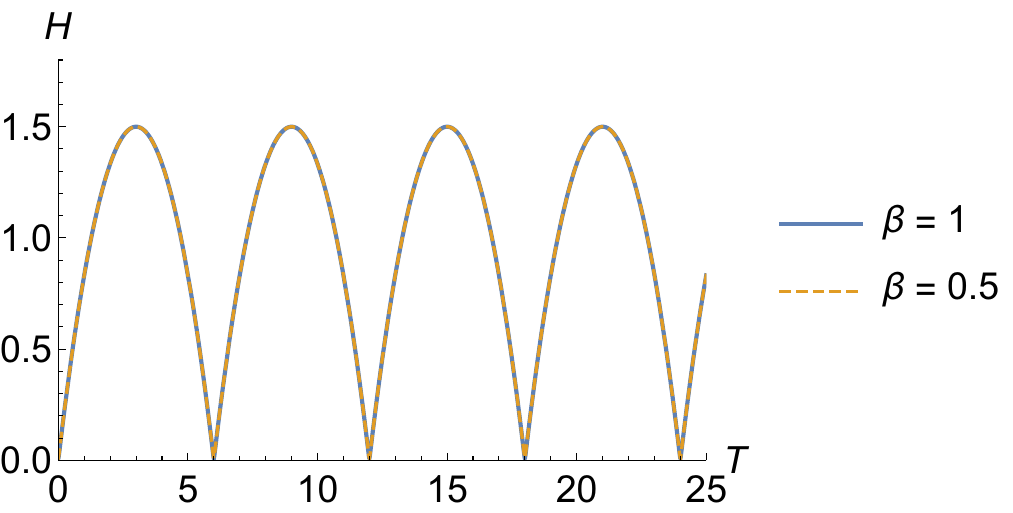}
		\caption{Case 4: \\ $a = 1/2$, $b = 0$,  $\omega \to \infty$}
		\label{fig:subim4}
	\end{subfigure}
	\caption{Flow regimes for different values of the slip parameter.}
	\label{fig:image2}
	\label{fig::FlowRegimes}
\end{figure}


Figure \ref{fig::FlowRegimes} illustrates the different flow regimes observed in our model, considering Cases 1--4 based on various reductions of the model \eqref{switala_scaled_with_imemersion1}. For each scenario, we show how the liquid height changes for two different values of the slip parameter: $\beta = 1$ (blue representing no-slip) and $\beta = 0.5$ (orange).

Cases 1 and 3 exclude gravitational effects and represent the phase in which the liquid begins to rise. Case 2 demonstrates how the viscous term stops the flow, leading to an equilibrium state with $H_{\mathrm{e}} = 1$. In all of these cases $\omega \to 0$, and the height of the liquid increases faster when the slip parameter is introduced (the slope is steeper compared to the case without slip).

In Case 4, however, the absence of viscosity leads to strong oscillatory behavior, with the solutions corresponding to slip and no-slip conditions overlapping.



\section{Global existence and uniqueness}
\label{sec::GlobalExistence}


The first step in the analysis of Washburn's equation will be to prove the existence and uniqueness of a solution. Washburn's equation is nonlinear, but its structure motivates a simple and natural transformation, 
stated in Lemma \ref{lema::transformisana_jednacina} below, that removes the nonlinearity from the leading term in the equation and relegates it to a lower-order term in the resulting transformed equation; this was already observed in \cite{plociniczak2018monotonicity}
in the special case of $\beta=1$. 
\begin{lemma} \label{lema::transformisana_jednacina}
	The initial-value problem
	\begin{equation}
		\omega(HH')' + \beta HH' + H = 1,
		\label{eq::H_jednacina}
	\end{equation}
	\begin{equation}
		H(0) = \alpha, \quad H'(0) = 0,
		\label{general_initialconditions_existence}
	\end{equation}
	can be transformed via $u(s) := \frac{1}{2}H(T)^{2}$ with $s := \frac{T}{\sqrt{\omega}}$ into an initial-value problem
	\begin{equation}
		u'' + \frac{\beta}{\sqrt{\omega}}u' + \sqrt{2u} = 1,
		\label{eq::transformisana_Ujednacina}
	\end{equation}
	\begin{equation}
		u(0) = \frac{1}{2}\alpha^{2}, \quad u'(0) = 0,
		\label{eq::transformisani_pocetniUslovi}
	\end{equation}
	with $\beta := \left(1 + \frac{4L}{R}\right)^{-1}$, $\alpha \geq 0$, and the independent variable $s \in [0,\infty)$.
\end{lemma}

The proof of the lemma is straightforward and is therefore omitted; for the special case of $\beta=1$ and $\alpha =0$ the reader may wish to consult Lemma 1 in \cite{plociniczak2018monotonicity}. 

Observe that the equation \eqref{eq::transformisana_Ujednacina} admits a stationary solution $u(s) = u_{e} = \mathrm{const.}$, where
\begin{equation}
	u_{e} = \frac{1}{2}, 
	\label{eq::Jurin_u}
\end{equation}
and corresponds to the stationary solutions \eqref{eq::Jurin_height} and \eqref{eq::Jurin_DLess}.

We shall prove the existence and uniqueness of a 
global-in-time solution to the initial-value problem \eqref{eq::transformisana_Ujednacina}, \eqref{eq::transformisani_pocetniUslovi} for all $\alpha \in [0,3/2]$. Our proof covers a more general set of initial data than the argument presented in \cite{plociniczak2018monotonicity}, which was restricted to zero initial data and to $\beta=1$.

\begin{theorem} For any $S>0$, $\beta>0$, $\omega>0$, and any $\alpha \in [0,3/2]$, the initial-value problem
	\eqref{eq::transformisana_Ujednacina}, \eqref{eq::transformisani_pocetniUslovi}
	has a unique solution $u \in C^2([0,S])$. When $\alpha=0$ the solution $u$ is strictly positive on $(0,S]$, and when $\alpha \in (0,3/2]$ the solution is strictly positive on $[0,S]$; furthermore $u(s)<9/8$ for all $s \in [0,S]$ and all $\alpha \in [0,3/2)$, and $u(s)\leq 9/8$ for all $s\in [0,S]$ and all $\alpha \in [0,3/2]$. Finally, for $\alpha \in [0,3/2]$ the solution depends continuously on the initial data; i.e., the initial-value problem is well-posed in the sense of Hadamard on the interval $[0,S]$ for all $S>0$ and for all $\alpha \in [0,3/2]$. 
	\label{th::global}
\end{theorem}

\begin{proof} 
	We begin by noting that since $u$ is, by definition, a nonnegative function (cf. Lemma \ref{lema::transformisana_jednacina}), if $u$ is a (nonnegative) solution to the initial-value problem \eqref{eq::transformisana_Ujednacina}, \eqref{eq::transformisani_pocetniUslovi} (assuming that such a $u$ exists), then it is automatically a solution to the initial-value problem
	\begin{equation}
		u'' + \frac{\beta}{\sqrt{\omega}}u' + \sqrt{2[u]_{+}} = 1,
		\label{eq::transformisana_Ujednacina1}
	\end{equation}
	\begin{equation}
		u(0) = \frac{1}{2}\alpha^{2}, \quad u'(0) = 0,
		\label{eq::transformisani_pocetniUslovi1}
	\end{equation}
	where $[u]_{+} := \operatorname{max}\{u,0\} = \frac{1}{2}(u+|u|)\geq 0$. Conversely, if $u$ is a solution of \eqref{eq::transformisana_Ujednacina1}, 
	\eqref{eq::transformisani_pocetniUslovi1} (assuming that such a $u$ exists), then it is straightforward to prove that $u$ must be nonnegative, i.e., $[u]_{+}=u$, and therefore $u$ is then a solution to the initial-value problem \eqref{eq::transformisana_Ujednacina}, 
	\eqref{eq::transformisani_pocetniUslovi}. Indeed, let $[u]_{-}:=\operatorname{min}\{u,0\}= \frac{1}{2}(u-|u|)\leq 0$. Multiplying \eqref{eq::transformisana_Ujednacina1} by $([u]_{-})'$ yields the equality
	\begin{align*}
		\frac{\dd}{\dd s}\left(\frac{1}{2}\big[([u]_{-})'\big]^2(s) - [u]_{-}(s) \right) = -\frac{\beta}{\sqrt{\omega}} \big[([u]_{-})'\big]^2(s) \leq 0,
	\end{align*}
	because $\sqrt{2[u]_{+}(s)} \, ([u]_{-})'(s)=0$. (Note that if $[u]_{+}(s)>0$ at some $s \in [0,S]$, then by continuity $[u]_{+}$ is strictly positive in some nonempty neighborhood of such an $s$; but then $[u]_{-}$ is identically zero in that same neighborhood, whereby $([u]_{-})'$ is also identically zero in that same neighborhood). Therefore, upon integration, 
	\begin{align}\label{eq::bd1}
		\frac{1}{2}\big[([u]_{-})'\big]^2(s) + (- [u]_{-}(s)) \leq \frac{1}{2}\big[([u]_{-})'\big]^2(0) + (- [u]_{-}(0)).
	\end{align}
	As $\alpha^2/2\geq 0$, it follows that $[u]_{-}(0)=0$. Furthermore, since $([u]_{-})'(s) = \frac{1}{2}(u'(s) - u'(s) \operatorname{sign}(u(s))$ and $u'(0)=0$, also $([u]_{-})'(0)=0$. Hence, the right-hand side of \eqref{eq::bd1} is equal to 0. As the left-hand side of \eqref{eq::bd1}, which is the sum of two nonnegative terms, is nonnegative, it must also be equal to zero. It then follows that in particular $(- [u]_{-}(s))=0$ for all $s \geq 0$ for which a solution to 
	\eqref{eq::transformisana_Ujednacina1}, 
	\eqref{eq::transformisani_pocetniUslovi1}  exists. This then implies that, if a solution to \eqref{eq::transformisana_Ujednacina1}, 
	\eqref{eq::transformisani_pocetniUslovi1} exists, it must be nonnegative. Hence, any solution to \eqref{eq::transformisana_Ujednacina1}, 
	\eqref{eq::transformisani_pocetniUslovi1} is a solution to \eqref{eq::transformisana_Ujednacina}, 
	\eqref{eq::transformisani_pocetniUslovi}. 
	
	Having shown that problems \eqref{eq::transformisana_Ujednacina}, 
	\eqref{eq::transformisani_pocetniUslovi} and \eqref{eq::transformisana_Ujednacina1}, 
	\eqref{eq::transformisani_pocetniUslovi1} are equivalent, and that if problem \eqref{eq::transformisana_Ujednacina1}, 
	\eqref{eq::transformisani_pocetniUslovi1} has a solution then this must be nonnegative, we shall focus on showing the existence of a unique solution to the latter on the entire interval $[0,S]$. We shall do so by first showing the existence of a unique solution to a regularized version of problem \eqref{eq::transformisana_Ujednacina1}, 
	\eqref{eq::transformisani_pocetniUslovi1} on the entire interval $[0,S]$. To this end, for $\varepsilon \in (0,1)$, we consider the following regularized initial-value problem:
	\begin{equation}
		\ue''+ \frac{\beta}{\sqrt{\omega}}\ue' + \sqrt{2[\ue]_{+} + \varepsilon} = 1,
		\label{eq::transformisana_Ujednacinae}
	\end{equation}
	\begin{equation}
		\ue(0) = \frac{1}{2}\alpha^{2}, \quad \ue'(0) = 0.
		\label{eq::transformisani_pocetniUslovie}
	\end{equation}
	As the function $z \in \mathbb{R} \mapsto \sqrt{2[z]_{+} + \varepsilon} \in \mathbb{R}_{\geq \varepsilon}$ is globally Lipschitz, the existence of a unique solution $\ue \in C^2([0,S])$ to the initial-value-problem \eqref{eq::transformisana_Ujednacinae}, \eqref{eq::transformisani_pocetniUslovie}, for any $S>0$, follows by a straightforward application of the Picard--Lindel{\"o}f theorem. 
	
	\smallskip
	
	We shall prove the \textit{existence} of a solution to the initial-value problem
	\eqref{eq::transformisana_Ujednacina1}, 
	\eqref{eq::transformisani_pocetniUslovi1} by passing to the limit $\varepsilon \to 0_+$ in the problem \eqref{eq::transformisana_Ujednacinae}, 
	\eqref{eq::transformisani_pocetniUslovie}.
	To be able to do this, we need suitable uniform bounds on the family of functions $(\ue)_{0<\varepsilon<1}$, which we shall now derive. By multiplying equation \eqref{eq::transformisana_Ujednacinae} by $\ue'$, we have that 
	\begin{align*}
		\frac{1}{2} \frac{\dd}{\dd s} [\ue'(s)]^2 + \frac{\beta}{\sqrt{\omega}}[\ue'(s)]^2 = \ue'(s) - \ue'(s) \sqrt{2[\ue(s)]_{+} + \varepsilon}.
	\end{align*}
	Hence, upon integration and using the elementary inequality $ab \leq \frac{\delta}{2}a^2 + \frac{1}{2\delta}b^2$ with $\delta = \beta/\sqrt{\omega}$, we find that
	\begin{align*}
		\frac{1}{2} [\ue'(s)]^2 + \frac{\beta}{\sqrt{\omega}}\int_0^s  [\ue'(t)]^2 \dd t &= \int_0^s \ue'(t) \dd t - \int_0^s \ue'(t) \sqrt{2[\ue(t)]_{+} + \varepsilon}\dd t\\
		&\leq \int_0^s |\ue'(t)| \dd t + \frac{\beta}{2\sqrt{\omega}}\int_0^s  [\ue'(t)]^2 \dd t + \frac{\sqrt{\omega}}{2\beta}\int_0^s (2[\ue(t)]_{+} + \varepsilon)\dd t.
	\end{align*}
	Thus, and because $[\ue(t)]_{+}  \leq |\ue(t)|$, it follows that
	\begin{align*}
		\frac{1}{2} [\ue'(s)]^2 + \frac{\beta}{2\sqrt{\omega}}\int_0^s  [\ue'(t)]^2 \dd t & \leq \int_0^s |\ue'(t)| \dd t + \frac{\sqrt{\omega}}{2\beta}\int_0^s (2|\ue(t)| + \varepsilon)\dd t.
	\end{align*}
	As $\ue(s) = \frac12\alpha^2 + \int_0^s \ue'(t) \dd t$, whereby $2|\ue(t)| \leq  \alpha^2 + 2\int_0^t |\ue'(\tau)| \dd \tau$ for all $t \in [0,s]\subseteq [0,S]$, we have that
	\begin{align*}
		\frac{1}{2} [\ue'(s)]^2 + \frac{\beta}{2\sqrt{\omega}}\int_0^s  [\ue'(t)]^2 \dd t  & \leq \int_0^s |\ue'(t)|\dd t + \frac{\sqrt{\omega}}{2\beta}\int_0^s \left(\alpha^2 + 2\int_0^t |\ue'(\tau)|\dd \tau + \varepsilon\right)\dd t\\
		& \leq \frac{s\sqrt{\omega}}{2\beta}\left(\alpha^2+ \varepsilon\right) + \left(1+ \frac{s\sqrt{\omega}}{\beta}\right) \int_0^s |\ue'(t)| \dd t\\
		& \leq \frac{s\sqrt{\omega}}{2\beta}\left(\alpha^2+ \varepsilon\right) + \left(1+ \frac{s\sqrt{\omega}}{\beta}\right) \sqrt{s}\left(\int_0^s [\ue'(t)]^2 \dd t\right)^{1/2}\\
		& \leq  \frac{s\sqrt{\omega}}{2\beta}\left(\alpha^2+ \varepsilon\right) + \frac{\sqrt{\omega}}{\beta}\left(1 + \frac{s \sqrt{\omega}}{\beta}\right)^2 s + \frac{\beta}{4\sqrt{\omega}}\int_0^s [\ue'(t)]^2 \dd t.
	\end{align*}
	Therefore, for any $S>0$ and all $s \in [0,S]$ and all $\varepsilon \in (0,1)$, 
	\begin{align*}
		\frac{1}{2} [\ue'(s)]^2 + \frac{\beta}{4\sqrt{\omega}}\int_0^s  [\ue'(t)]^2 \dd t  \leq
		\frac{S\sqrt{\omega}}{2\beta}\left(\alpha^2+ 1 \right) + \frac{\sqrt{\omega}}{\beta}\left(1 + \frac{S \sqrt{\omega}}{\beta}\right)^2 S.
	\end{align*}
	This then implies that
	\begin{align}\label{eq::BD1} \operatorname{max}_{s \in [0,S]}|\ue'(s)| \leq C,
	\end{align}
	where $C$ denotes a generic positive constant independent of $\varepsilon \in (0,1)$; therefore, again because  $\ue(s) = \alpha^2/2 + \int_0^s \ue'(t) \dd t$, we also have that 
	\[ \operatorname{max}_{s \in [0,S]}|\ue(s)| \leq C.\]
	Thus we infer that the family of functions $(\ue)_{0<\varepsilon<1} \subset C^2([0,S])$ is uniformly bounded in $C([0,S])$ and that it is uniformly equicontinuous. Indeed, by the mean value theorem and \eqref{eq::BD1}, $|\ue(s_1) - \ue(s_2)| \leq C|s_1 - s_2|$ for all $s_1, s_2 \in [0,S]$, uniformly in $\varepsilon \in (0,1)$, so the uniform equicontinuity of the family directly follows. Hence, by the Arzel\`{a}--Ascoli theorem, there exists a subsequence of $(\ue)_{0<\varepsilon<1}$ (not explicitly indicated), which uniformly converges on $[0,S]$ to a continuous function $u \in C([0,S])$ as $\varepsilon \to 0_+$. Next, we shall show that $u$ is in fact a solution to the initial-value problem \eqref{eq::transformisana_Ujednacina1}, 
	\eqref{eq::transformisani_pocetniUslovi1}. We shall do so by restating the initial-value problem \eqref{eq::transformisana_Ujednacinae}, 
	\eqref{eq::transformisani_pocetniUslovie} in an equivalent form as a nonlinear Volterra integral equation of the second kind, and passing to the limit $\varepsilon \to 0_+$ in this integral equation. 
	
	By integrating \eqref{eq::transformisana_Ujednacinae} from $0$ to $s \in (0,S]$, using $\mathrm{e}^{\beta s/\sqrt{\omega}}$ as an integrating factor, and then integrating the resulting equality again, it follows that $\ue$ satisfies the following nonlinear Volterra integral equation:
	\begin{equation*}
		\ue(s) = \frac{\alpha^2}{2} + \int_0^s \frac{\sqrt{\omega}}{\beta}\bigg[1-\mathrm{e}^{\frac{-\beta (s-t)}{\sqrt{\omega}}}\bigg] \bigg(1 - \sqrt{2[\ue(t)]_{+} + \varepsilon}\bigg)\dd t, \quad s \in [0,S]. 
	\end{equation*}
	By passing to the limit $\varepsilon \to 0_+$ in this equation over the uniformly convergent subsequence whose existence has been guaranteed above by the Arzel\`{a}--Ascoli theorem, it follows that the limiting function $u$ satisfies
	\begin{equation}\label{eq::AA}
		u(s) = \frac{\alpha^2}{2} + 
		\int_0^s \frac{\sqrt{\omega}}{\beta}\bigg[1-\mathrm{e}^{\frac{-\beta (s-t)}{\sqrt{\omega}}}\bigg] \bigg(1 - \sqrt{2[u(t)]_{+} }\bigg)\dd t, \quad s \in [0,S]. 
	\end{equation}
	Clearly, $u(0)=\alpha^2/2$. As $u \in C([0,S])$, the right-hand side of \eqref{eq::AA} belongs to $C^1([0,S])$, whereby the function $u$ appearing on the left-hand side of \eqref{eq::AA} also belongs to $C^1([0,S])$. We are therefore allowed to differentiate the equality \eqref{eq::AA}, and we find that 
	\begin{equation}\label{eq::AA+} u'(s) = \mathrm{e}^{-\frac{\beta s}{\sqrt{\omega}}}\int_0^s  \mathrm{e}^{\frac{\beta t}{\sqrt{\omega}}} (1- \sqrt{2[u(t)]_{+}}) \dd t.
	\end{equation}
	Obviously $u'(0)=0$. Since the integrand appearing in \eqref{eq::AA+} belongs to $C([0,S])$, and therefore the right-hand side of \eqref{eq::AA+} belongs to $C^1([0,S])$, the same must be true of the left-hand side of \eqref{eq::AA+}; hence, $u \in C^2([0,S])$. We are therefore allowed to differentiate \eqref{eq::AA+}; using the expression for $u'(s)$ stated in \eqref{eq::AA+} in the equality resulting from differentiating \eqref{eq::AA+}, we then find that  
	\[u''(s)= - \frac{\beta}{\sqrt{\omega}}u'(s) + (1 - \sqrt{2[u(s)]_{+}}). \]
	Thus we have shown that the function $u$ satisfies the differential equation \eqref{eq::transformisana_Ujednacina1} and the initial conditions \eqref{eq::transformisani_pocetniUslovi1}. In other words, we have proved the existence of a solution $u \in C^2([0,S])$ to the initial-value problem \eqref{eq::transformisana_Ujednacina1}, 
	\eqref{eq::transformisani_pocetniUslovi1} on any interval $[0,S]$. As we have already shown that all solutions to this initial-value problem are nonnegative, it follows that we have proved the existence of a nonnegative solution $u \in C^2([0,S])$ to the initial-value problem \eqref{eq::transformisana_Ujednacina}, 
	\eqref{eq::transformisani_pocetniUslovi}
	on any interval $[0,S]$. 
	
	\smallskip
	
	Let us now show that all solutions to the initial-value problem \eqref{eq::transformisana_Ujednacina}, 
	\eqref{eq::transformisani_pocetniUslovi} must, in fact, be strictly positive throughout the interval $[0,S]$ in the case of $\alpha \in (0,3/2]$, and over the half-open interval $(0,S]$ in the case of $\alpha=0$. Having done so, we shall be able to prove uniqueness of the solution. By multiplying \eqref{eq::transformisana_Ujednacina} by $u'(s)$ it follows that 
	\begin{align*}
		\frac{\dd}{\dd s}\left[\frac{1}{2} [u'(s)]^2 - u(s) + \frac{2\sqrt{2}}{3}[u(s)]^{3/2}\right] =  - \frac{\beta}{\sqrt{\omega}}[u'(s)]^2.
	\end{align*}
	Letting 
	\[ E(s):= \frac{1}{2}[u'(s)]^2 - u(s) + \frac{2\sqrt{2}}{3}[u(s)]^{3/2},\]
	it follows that $E$ is a monotonically nonincreasing function of $s$ and 
	\[ E(0) = -\frac{\alpha^2}{2}\bigg(1- \frac{2\alpha}{3}\bigg).\]
	We will first show that $E(s) < 0$ for all $s \in (0, S]$ and all $\alpha \in [0, 3/2]$.
	
	\smallskip
	
	Clearly, if $\alpha \in (0,3/2)$ then $E(0)<0$, and therefore $E(s)<0$ for all $s \in [0,S]$, because $E$ is monotonically nonincreasing.

	If,  on the other hand, $\alpha = 0$, then $E(0)=0$ and since $E(s)$ is a nonincreasing function of $s$, it follows that $E(s)\leq 0$ for all $s \in [0,S]$. We will show that, in fact, $E(s)<0$ for all $s \in (0,S]$. Since $u \in C^2([0,S]$) and 
	\[ u''(s) = - \frac{\beta}{\sqrt{\omega}}u'(s) + (1 - \sqrt{2u(s)}),\]
	by passing to the limit $s \to 0_+$ and recalling that $u(0)=0$ and $u'(0)=0$, it follows that $u''(0)=1$. Since $u'' \in C([0,S])$, we deduce that $u''$ must be strictly positive in a small neighborhood of $s=0$, which means that $u'$ is a strictly monotonically increasing function in that neighborhood. As $u'(0)=0$, it then follows that $u'(s)>0$ for all $s>0$ in that small neighborhood of $s=0$. Because $E(0)=0$ and $E(s)\leq 0$, $E(s)$ cannot be identically equal to zero over this small neighborhood of $s=0$, as in that case we would have that 
	\[ \frac{1}{2}[u'(s)]^2 - u(s) + \frac{2\sqrt{2}}{3}[u(s)]^{3/2} = 0\]
	over that interval, and by differentiating this equality (recall that $u \in C^2([0,S])$) it would follow that 
	\[ u'(s) u''(s) - u'(s) + \sqrt{2}[u(s)]^{1/2} u'(s) = 0 \]
	for all such $s$. However, $u'(s)>0$ for all such $s$, and therefore it would follow that 
	\[ u''(s) - 1 + \sqrt{2}[u(s)]^{1/2} = 0.\]
	Substitution of $u''(s) = 1 - \sqrt{2}[u(s)]^{1/2}$ into the differential equation
	\eqref{eq::transformisana_Ujednacina} would then imply that $(\beta/\sqrt{\omega})u'(s)=0$ on that interval, which is impossible because $u'(s)>0$ there. Thus, we deduce that $E(s)<0$ for all $s>0$ sufficiently close to $s=0$. However, since $E$ is nonincreasing, we infer that $E(s)<0$ for all $s \in (0,S]$.
	
	If $\alpha =3/2$, then $E(0)=0$ and since $E(s)$ is a nonincreasing function of $s$, it follows that $E(s)\leq 0$ for all $s \in [0,S]$. We will show by an argument similar to the one presented in the case of $\alpha=0$ that, in fact, $E(s)<0$ for all $s \in (0,S]$. Since $u \in C^2([0,S]$) and 
	\[ u''(s) = - \frac{\beta}{\sqrt{\omega}}u'(s) + (1 - \sqrt{2u(s)}),\]
	by passing to the limit $s \to 0_+$ and recalling that $u(0)=\alpha^2/2 = 9/8$ and $u'(0)=0$, it follows that $u''(0)=-1/2$. Since $u'' \in C([0,S])$, we deduce that $u''$ must be strictly negative in a small neighborhood of $s=0$, which means that $u'$ is a strictly monotonically decreasing function in that neighborhood. As $u'(0)=0$, it then follows that $u'(s)<0$ for all $s>0$ in that small neighborhood of $s=0$. Because $E(0)=0$, $E(s)\leq 0$, and $E(s)$ cannot be identically equal to zero in this small neighborhood of $s=0$ (by the same argument as in the case of $\alpha=0$ above, but now using $u'(s)<0$ for all $s>0$ in this small neighborhood of $s=0$), the fact that $E$ is nonincreasing implies that $E(s)<0$ for all $s$ sufficiently close to $s=0$. However, since $E$ is nonincreasing on the interval $[0,S]$, it then follows that $E(s)<0$ for all $s \in (0,S]$.
	
	Thus we have shown that, for all $\alpha \in [0,3/2]$, $E(s)<0$ for all $s \in (0,S]$. Hence,
	\[ 0 > E(s) =  \frac{1}{2}[u'(s)]^2 - u(s) + \frac{2\sqrt{2}}{3}[u(s)]^{3/2} >  -u(s) + \frac{2\sqrt{2}}{3}[u(s)]^{3/2} = - u(s) \bigg[1 - \sqrt{\frac{8 u(s)}{9}}\bigg]\]
	for all $s \in (0,S]$. Consequently (recall that $u(s) \geq 0$), we have that 
	\begin{equation} \label{eq::EnergyEstimate}
		0 < u(s) < \frac{9}{8} \qquad \mbox{for all $s \in (0,S]$}.
	\end{equation}
	Thus we have shown that the initial-value problem 
	\eqref{eq::transformisana_Ujednacina}, 
	\eqref{eq::transformisani_pocetniUslovi} has a solution $u \in C^2([0,S])$ on any interval $[0,S]$ of the real line and that all solutions of this initial-value problem must be positive on $(0,S]$ when $\alpha=0$ and on the whole of $[0,S]$ when $\alpha \in (0,3/2]$; furthermore, all solutions must satisfy the upper bound $u(s) \leq  9/8$ for all $s \in [0,S]$ (thanks to \eqref{eq::EnergyEstimate} and because $u(0)=\alpha^2/2 \leq 9/8$ for all $\alpha \in [0,3/2]$).
	
	\smallskip
	
	It remains to prove the \textit{uniqueness} of the solution to the initial-value problem \eqref{eq::transformisana_Ujednacina}, 
	\eqref{eq::transformisani_pocetniUslovi}. Let us first consider the case where $\alpha \in (0,3/2]$. Suppose that $u$ and $v$ are two solutions to the initial-value problem \eqref{eq::transformisana_Ujednacina}, 
	\eqref{eq::transformisani_pocetniUslovi}. As both must be positive functions on the closed interval $[0,S]$, because of the continuity of $u$ and $v$ each of these functions can only have a finite number of minima on $[0,S]$ and at each of these minima both $u$ and $v$ must be strictly positive. Thus, there exists a positive real number $\delta$ such that $u(s)\geq \delta$ and $v(s)\geq \delta$ for all $s \in [0,S]$. Equation \eqref{eq::AA} then implies that
	\[ |u(s) - v(s)| \leq  \frac{\sqrt{2\omega}}{\beta}\int_0^s |\sqrt{u(t)} - \sqrt{v(t)}| \dd t\leq  \frac{\sqrt{2\omega}}{\beta} \frac{1}{2\sqrt{\delta}}
	\int_0^s |u(t) - v(t)| \dd t\]
	for all $s \in [0,S]$. It then follows from this inequality by Gr\"onwall's lemma that $|u(s) - v(s)|=0$ for all $s \in [0,S]$. Thus we have shown the uniqueness of the solution when $\alpha \in (0,3/2]$.

	Now suppose that $\alpha=0$. Our aim is to apply Theorem 11.3 of Precup \cite{precup2002methods},\footnote{For the 
		reader's convenience, we include here the statement of this result (in an abbreviated form).\\ \textbf{Theorem 11.3} (of \cite{precup2002methods}, p.171) Let $(X,K)$ be an ordered Banach space (i.e., a Banach space endowed with a closed cone $K$) and let $[u_0,v_0]:=(u_0+K)\cap (v_0-K)=\{v \in X\,:\, u_0 \leq v \leq v_0\} \subset X$ be an \textit{order interval} with $0 \leq u_0 \leq v_0$ such that there is a real number $\lambda$ with
		\[ \lambda >0, \quad \lambda v_0 \leq u_0.\]
		Assume that $T\,:\,[0,v_0] \to X$ is decreasing on $[0,v_0]$, continuous on $[u_0,v_0]$, $T([u_0,v_0]) \subset [u_0,v_0]$, and that there is a function $\chi\,:\, (0,1) \to (0,\infty)$ such that
		\[ \chi(\lambda) < \frac{1}{\lambda}, \quad T(\lambda v) \leq  \chi(\lambda) T(v)\]
		for all $\lambda \in (0,1)$, $v \in [u_0,v_0]$. In addition assume that $K$ is a normal cone and $T$ is compact. Then, $T$ has a unique fixed point $u$ in $[u_0,v_0]$.} 
	to deduce the existence of a \textit{unique} (local) solution to 
	the initial-value problem on the closed interval $[0,s_*]$, where $s_*:= \operatorname{min}\{1/2,\sqrt{\omega}/\beta\}$. To be able to do so, we need to check the assumptions of that theorem. 
	\begin{itemize}
		\item
		We begin by noting that $X:=C([0,s_*])$ is an ordered Banach space with the ordering $\leq$ (i.e., given $v, w \in X$, we say that $v \leq w$ if, and only if, $v(s) \leq w(s)$ for all $s \in [0,s_*]$) and 
		\[K:=\{ v \in X\,:\, v(s) \geq 0\quad \forall\, s \in [0,s_*]\}\] 
		is a normal cone in $X$. We recall that a cone $K \subset X$ is normal if, for any pair of sequences $(w_k)_{k=1}^\infty \subset K$ and $(v_k)_{k=1}^\infty \subset K$ such that $\lim_{k\to \infty} v_k = 0$ in $X$ and $0\leq w_k \leq v_k$, it follows that $\lim_{k\to \infty} w_k= 0$ in $X$.
		That our $K$, as defined above, is a cone in $X$ is trivially true, as any nonnegative constant multiple of a nonnegative continuous function defined on the interval $[0,s_*]$ is a nonnegative continuous function defined on the interval $[0,s_*]$. To prove that our cone $K$ is normal, we note that if $(v_k)_{k=1}^\infty \subset K$ converges to $0$ in $X$
		(i.e., it converges to 0 uniformly on the interval $[0,s_*]$), then, because $0 \leq \operatorname{max}_{s \in [0,s_*]} w_k(s) \leq \operatorname{max}_{s \in [0,s_*]} v_k(s) \to 0$ as $k \to \infty$, it follows that 
		the sequence $(w_k)_{k=1}^\infty \subset K$ must also converge to $0$ uniformly, i.e., $\lim_{k\to \infty} w_k = 0$ in $X$.

		\item Next, for $s \in [0,s_*]$, we let $u_0(s):=s^2/6$ and $v_0(s):=s^2/2$
		and we take $[u_0,v_0]$ as the order interval in Precup's Theorem 11.3. Because the function $z \in \mathbb{R}_{\geq 0} \mapsto 1 - \sqrt{2z} \in \mathbb{R}_{\leq 1}$ is 
		decreasing and $0 \leq 1 - \mathrm{e}^{-z} \leq 1$ for all $z \geq 0$, it follows that the
		map $T\,:\, [0,v_0] \to X$ defined by
		\[[T(v)](s):=\frac{\sqrt{\omega}}{\beta} \int_0^s \bigg[1-\mathrm{e}^{\frac{-\beta (s-t)}{\sqrt{\omega}}}\bigg] \bigg(1 - \sqrt{2v(t) }\bigg)\dd t \]
		is decreasing, in the sense that if $0 \leq  w \leq v$ on $[0,s_*]$ then $T(0) \geq T(w)\geq T(v)$ on $[0,s_*]$. 
		
		\item Furthermore, the map $T: [u_0,v_0] \to X$ is continuous. Indeed, if $u_0 \leq v_k  \leq v_0$ and $v_k \to v$ in $X$, then $u_0 \leq v \leq v_0$; also, $T(v_k) \to T(v)$ in $X$, thanks to the inequality $|\sqrt{2v_k(t)} - \sqrt{2v(t)}| \leq \sqrt{2}|v_k(t) - v(t)|^{1/2}$, which implies that
		$\operatorname{max}_{s \in [0,s_*]}|[T(v_k)](s)-[T(v)](s)| \leq (S/\sqrt{2}) \operatorname{max}_{s \in [0,s_*]}|v_k(s)-v(s)|^{1/2} \to 0$ as $k \to \infty$.
		
		\item As a matter of fact, the map $T: [u_0,v_0] \to X$ is compact, i.e.,  if $(v_k)_{k=1}^\infty$ is a sequence in $X=C([0,s_*])$, such that $u_0 \leq v_k \leq v_0$, then the mapped sequence $(T(v_k))_{k=1}^\infty$ has a convergent subsequence in $X$. Indeed, as $T(v_0) \leq T(v_k) \leq T(u_0)$ for $k=1,2,\ldots$, the sequence $(T(v_k))_{k=1}^\infty$ is bounded in $X$. Furthermore, as
		\begin{equation*} [T(v_k)]'(s) = \mathrm{e}^{-\frac{\beta s}{\sqrt{\omega}}}\int_0^s  \mathrm{e}^{\frac{\beta t}{\sqrt{\omega}}} (1- \sqrt{2v_k(t)}) \dd t,
		\end{equation*}
		it follows that $T(v_k) \in C^1([0,s_*])$ and $(T(v_k))_{k=1}^\infty$ is a bounded sequence in $C^1([0,s_*])$. Thanks to the compact embedding of $C^1([0,s_*])$ in $C([0,s_*])=X$ (or by the Arzel\`a--Ascoli theorem) it then follows that there exists a subsequence $(T(v_{k_n}))_{n=1}^\infty$
		that converges in $X$ to an element of $X$.  Hence, the map $T$ is compact.\footnote{In fact, in the statement of Theorem 11.3 of \cite{precup2002methods} the author uses the term \textit{completely continuous map} in reference to $T$, but his notion of completely continuous map is what is now called \textit{compact map}, namely one that maps a bounded sequence into a sequence that has a norm-convergent subsequence in the target space. This is a stronger requirement than being a completely continuous map, i.e., one that maps a weakly convergent sequence into a norm-convergent sequence in the target space.}
		
		\item Next, a simple calculation yields that with our choices of $u_0(s)=s^2/6$ and $v_0(s)=s^2/2$, we have that $u_0 \leq T(v_0)$ and $T(u_0) \leq v_0$ on $[0,s_*]$. 
		Let us confirm that this is indeed the case. Let $\gamma:={\beta}/{\sqrt{\omega}}$. Because $1-\mathrm{e}^{-x} \geq x - \frac{1}{2}x^2$ for all $x \geq 0$ and $0 \leq s \leq s_* = \operatorname{min}\{1/2,\sqrt{\omega}/{\beta}\} = \operatorname{min}\{1/2,1/\gamma\}\leq 1/2$, we have that 
		\begin{align*}
			[T(v_0)](s) &= \frac{1}{\gamma}\int_0^s (1-\mathrm{e}^{-\gamma(s-t)})(1-t)\dd t \\
			& \geq \frac{1}{2\gamma}\int_0^s \left(\gamma (s-t) - \frac{\gamma^2(s-t)^2}{2}\right) \dd t = \frac{s^2}{4}\left(1-\frac{\gamma s}{3}\right) \geq \frac{s^2}{6} = u_0(s)\quad \forall\, s \in [0,s_*].
		\end{align*}
		On the other hand, since $1-\mathrm{e}^{-x} \leq x$ for all $x \geq 0$, it follows that
		\begin{align*}
			[T(u_0)](s) &= \frac{1}{\gamma}\int_0^s (1-\mathrm{e}^{-\gamma(s-t)})\left(1-\frac{t}{\sqrt{3}}\right)\dd t \\
			& \leq \frac{1}{\gamma}\int_0^s \gamma(s-t)\left(1-\frac{t}{\sqrt{3}}\right)\dd t = \frac{s^2}{2}\left(1-\frac{s}{3\sqrt{3}}\right) \leq \frac{s^2}{2}=v_0(s) \quad \forall\, s \in [0,s_*].
		\end{align*}

		Hence, thanks to the monotonicity of $T$, it follows that $T([u_0,v_0]) \subset [u_0,v_0]$. 
		Indeed, if $u_0 \leq w \leq v_0$, then $u_0 \leq T(v_0) \leq T(w) \leq T(u_0) \leq v_0$.
		
		\item Finally, we observe that, for any $\lambda \in (0,1)$, $v \in [u_0,v_0]$ and $t \in [0,s_*]$, we have that
		\[ 1 - \sqrt{2 \lambda v(t)} \leq \frac{1}{\sqrt{\lambda}}\left(1- \sqrt{2v(t)}\right),\]
		because
		\[ \sqrt{2v(t)}\,(1-\lambda) \leq t\, (1-\lambda) \leq \frac{1}{2} (1-\lambda) = \frac{1}{2}(1+\sqrt{\lambda})(1-\sqrt{\lambda}) \leq 1 - \sqrt{\lambda}.\]
		Consequently, 
		\begin{align*}[T(\lambda v)](s)=\frac{\sqrt{\omega}}{\beta} \int_0^s \bigg[1-\mathrm{e}^{\frac{-\beta (s-t)}{\sqrt{\omega}}}\bigg] \bigg(1 - \sqrt{2\lambda v(t) }\bigg)\dd t \leq \frac{1}{\sqrt{\lambda}} [T(v)](s) \quad \forall\, s \in [0,s_*].
		\end{align*}
		Clearly the function $\chi\,:\,(0,1) \to (0,\infty)$ defined by $\chi(\lambda):=1/\sqrt{\lambda}$ satisfies $\chi(\lambda)<1/\lambda$ for all $\lambda \in (0,1)$ and $T(\lambda v) \leq \chi(\lambda) T(v)$ for all $\lambda \in (0,1)$ and all $v \in [u_0,v_0]$.
	\end{itemize}
	
	Thus, we have checked all the hypotheses of Theorem 11.3 of \cite{precup2002methods}, which then implies that $T$ has a unique fixed point $u \in K$. In other words, there exists a \textit{unique} (local) solution $u$ to the initial-value problem on the interval $[0,s_*]$ when $\alpha=0$.
	
	\smallskip
	
	If $S \in (0,s_*]$, then the proof of the uniqueness (and existence) of a solution $u$ in the case of $\alpha=0$ is complete. If, on the other hand, $S>s_*$, then we proceed as follows. Let us define $u_*:=u(s_*)$ and $u'_*:=u'(s_*)$. Then, it remains to show uniqueness over the interval $[s_*,S]$ of solutions whose existence over $[0,S]$ we have already established above. Suppose therefore that $u$ and $v$ are two solutions on $[0,S]$ to the initial-value problem \eqref{eq::transformisana_Ujednacina}, 
	\eqref{eq::transformisani_pocetniUslovi} for $\alpha=0$. As $u(s)$ must be equal to $v(s)$ for all $s \in [0,s_*]$ thanks to the uniqueness of the solution on $[0,s_*]$, it follows that $u(s_*)=v(s_*)=u_*$ and $u'(s_*)=v'(s_*)=u'_*$. By integrating the differential equation \eqref{eq::transformisana_Ujednacina}, now from $s_*$ to $s \in (s_*,S]$, we then have
	\begin{equation}\label{eq::AA2}
		u(s) = u_* + \frac{\sqrt{\omega}}{\beta}\bigg(1 - \mathrm{e}^{\frac{\beta (s-s_*)}{\sqrt{\omega}}}\bigg)u_*' + 
		\int_{s_*}^s \frac{\sqrt{\omega}}{\beta}\bigg[1-\mathrm{e}^{\frac{-\beta (s-t)}{\sqrt{\omega}}}\bigg] \bigg(1 - \sqrt{2u(t) }\bigg)\dd t
	\end{equation}
	and
	\begin{equation}\label{eq::AA3}
		v(s) = u_* + \frac{\sqrt{\omega}}{\beta}\bigg(1 - \mathrm{e}^{\frac{\beta (s-s_*)}{\sqrt{\omega}}}\bigg)u_*' + 
		\int_{s_*}^s \frac{\sqrt{\omega}}{\beta}\bigg[1-\mathrm{e}^{\frac{-\beta (s-t)}{\sqrt{\omega}}}\bigg] \bigg(1 - \sqrt{2v(t) }\bigg)\dd t.
	\end{equation}
	As both $u$ and $v$ must be strictly positive on the entire closed interval $[s_*,S]$,  it follows from \eqref{eq::AA2} and \eqref{eq::AA3} by Gr\"onwall's lemma, in exactly the same way as above, that $|u(s)-v(s)|=0$ for all $s \in [s_*,S]$. Thus we have also shown the existence of a unique solution on the entire interval $[0,S]$ in the case of $\alpha=0$.
	
	\smallskip
	
	It remains to prove the continuous dependence of the solution on $\alpha$. Let us denote by $u_\alpha$ the unique solution of the initial-value problem corresponding to $\alpha \in [0,3/2]$.
	As $0\leq u_\alpha(s)\leq 9/8$ for all $s \in [0,S]$ and all $\alpha \in [0,3/2]$, it follows from \eqref{eq::AA+} that $|u_\alpha'(s)| \leq S$ for all $s \in [0,S]$ and all $\alpha \in [0,3/2]$. Let $\alpha_0 \in [0,3/2]$. Then, thanks to the Arzel\`a--Ascoli theorem, there exists a subsequence of the family $(u_\alpha)_{0\leq\alpha\leq 3/2}$, which converges to a function $\widehat{u}_{\alpha_0} \in C([0,S])$ uniformly on $[0,S]$ as $\alpha \to \alpha_{0} \in [0,3/2]$. Since $\widehat{u}_{\alpha_0}$ is the limit of a uniformly convergent sequence of nonnegative continuous functions, it too must be nonnegative on $[0,S]$. We then return to \eqref{eq::AA}, with $u$ replaced there by $u_\alpha$ and pass to the limit over this uniformly convergent subsequence as $\alpha \to \alpha_0$ to find that
	\begin{equation*}
		\widehat{u}_{\alpha_0}(s) = \frac{\alpha_0^2}{2}+
		\int_0^s \frac{\sqrt{\omega}}{\beta}\bigg[1-\mathrm{e}^{\frac{-\beta (s-t)}{\sqrt{\omega}}}\bigg] \bigg(1 - \sqrt{2\,\widehat{u}_{\alpha_0}(t) }\bigg)\dd t.
	\end{equation*}
	Clearly, $\widehat{u}_{\alpha_0}(0)=\alpha_0^2/2$, $\widehat{u}_{\alpha_0}'(0)=0$, and 
	\[\widehat{u}_{\alpha_0}''(s)= - \frac{\beta}{\sqrt{\omega}}\widehat{u}_{\alpha_0}'(s) + (1 - \sqrt{2\,\widehat{u}_{\alpha_0}(s)}). \]
	Therefore, the limiting function $\widehat{u}_{\alpha_0}$ is in fact the unique solution $u_{\alpha_0}$ of the initial value problem corresponding to $\alpha=\alpha_0 \in [0,3/2]$. As the initial-value problem corresponding to $\alpha_0$ has a unique solution, the same convergence argument applies to any subsequence of $(u_\alpha)_{0 \leq \alpha \leq 3/2}$ with $\alpha \to \alpha_0$. In other words, ${u}_{\alpha}$ converges to 
	$u_{\alpha_0}$ in the norm of $C([0,S])$ whenever $\alpha \to \alpha_0$.
	Thus, we have shown continuous dependence of the solution $u_\alpha$ on $\alpha \in [0,3/2]$.
	
	That completes the proof of the existence of a unique solution $u \in C^2([0,S])$ to the initial-value problem \eqref{eq::transformisana_Ujednacina}, 
	\eqref{eq::transformisani_pocetniUslovi} for all $\alpha \in [0,3/2]$ on the interval $[0,S]$ for any $S>0$, and of the well-posedness of the initial-value problem in the sense of Hadamard. 
\end{proof}

Although we have proved the well-posedness of the initial-value problem \eqref{eq::transformisana_Ujednacina}, 
\eqref{eq::transformisani_pocetniUslovi} on the interval $[0,S]$ for any $S>0$ and for all $\alpha \in [0,3/2]$, in view of our assumption \eqref{eq::InitCond_Dimensional}
and the definition of $\alpha$ as the ratio of the nonnegative initial height $h_0$ and the positive equilibrium height $h_e$ (cf. equation \eqref{eq::dimensionless}, the values of $\alpha$ that are of actual physical interest in the context of capillary rise are those in the range $0\leq \alpha \ll 1$. For this reason, in particular, we do not consider here the case where $\alpha > 3/2$.



\section{Stability and convergence to the equilibrium solution}
\label{sec::Stability}


We proceed by restating the second-order ordinary differential equation \eqref{eq::transformisana_Ujednacina} in an equivalent form as the following autonomous system of two first-order ordinary differential equations:
\begin{equation}
	\begin{split}
		u' &= v =: f_{1}(u,v), \\
		v' &= 1 - \frac{\beta}{\sqrt{\omega}}v - \sqrt{2u} =: f_{2}(u,v), 
	\end{split}
	\label{eq::sistem_U}
\end{equation}
with $f_1$ and $f_2$ understood to be functions of $(u,v)$ defined on $\Omega := \left\{ (u,v) \in \mathbb{R}^2 : u \geq 0,\, v \in \mathbb{R} \right\}$. The critical point of the system \eqref{eq::sistem_U} is the stationary solution $(u(s),v(s)) = (u_{\mathrm{cr}}, v_{\mathrm{cr}})$ for which both $f_{1}(u_{\mathrm{cr}}, v_{\mathrm{cr}}) = 0$ and $f_{2}(u_{\mathrm{cr}}, v_{\mathrm{cr}}) = 0$. Clearly, the system \eqref{eq::sistem_U} has a unique critical point
\begin{equation}
	(u_{\mathrm{cr}}, v_{\mathrm{cr}}) 
	= \left( \frac{1}{2}, 0 \right), 
	\label{eq::critical_U}
\end{equation}
which corresponds to the stationary solution \eqref{eq::Jurin_height} (or, equivalently, \eqref{eq::Jurin_DLess} or \eqref{eq::Jurin_u}). 

In \cite{plociniczak2018monotonicity} it was shown that the stationary solution of \eqref{eq::H_jednacina}, $H(T) = H_{0} = 1$, in the case without slip ($\beta = 1$) is reached monotonically or oscillatorily depending on the value of $\omega$. This stationary solution corresponds to the critical point \eqref{eq::critical_U} of the system \eqref{eq::sistem_U}. 

It turns out that when the slip parameter $\beta$ is introduced there are no qualitative changes in the asymptotic behavior of the solution---the only difference is in the critical value $\omega_\ast$ (cf.~equation \eqref{eq::crit} below) of the parameter $\omega$ that separates monotonic asymptotic behavior from oscillatory asymptotic behavior. To prove this claim, we shall linearize the system \eqref{eq::sistem_U} in the neighborhood of the critical point \eqref{eq::critical_U} by defining $x(s) := u(s) - u_{\mathrm{cr}}$, $y(s) := v(s) - v_{\mathrm{cr}}$, and note that the `perturbations' $x(s)$ and $y(s)$ satisfy the following autonomous linear system of differential equations: 
\begin{equation}
	\begin{split}
		x' &= y, \\
		y' &= - x - \frac{\beta}{\sqrt{\omega}} y.
	\end{split}
	\label{eq::sistem_Ulinear}
\end{equation}
We shall then analyze the stability and the type of the critical point \eqref{eq::critical_U}, or equivalently $(x_{\mathrm{cr}}, y_{\mathrm{cr}}) = (0,0)$.
We note in passing that $f_1$ and $f_2$ are $C^\infty$ functions in a small neighborhood of the critical point $(1/2,0)$;  therefore, the local linearization of the nonlinear system \eqref{eq::sistem_U} in the vicinity of the critical point is meaningful. 

In the following, the arguments largely proceed along similar lines as those in Section 3.2 of \cite{plociniczak2018monotonicity}.

\begin{theorem}\label{thm:lin}
	For the linear(ized) system \eqref{eq::sistem_Ulinear} the following assertions hold: 
	\begin{itemize}
		\item[(i)] the critical point $(x_{\mathrm{cr}}, y_{\mathrm{cr}}) = (0,0)$ is linearly asymptotically stable; 
		\item[(ii)] the type of the critical point depends on the values of the parameters, i.e., there is a critical value of $\omega$ 
		\begin{equation}\label{eq::crit}
			\omega_{\ast} = \frac{\beta^2}{4}, 
		\end{equation}
		such that 
		\begin{itemize}
			\item[(a)] for $\omega < \omega_{\ast}$ the critical point is a stable node;
			\item[(b)] for $\omega > \omega_{\ast}$ the critical point is a stable spiral;
			\item[(c)] for $\omega = \omega_{\ast}$ the critical point is a stable inflected node.
		\end{itemize}
	\end{itemize}
	\label{prop::spektralnaAnaliza_slip}
\end{theorem}

\begin{proof}
	The matrix 
	\[ A:=\left(\begin{array}{cc} 0 & 1 \\
		-1 & - \frac{\beta}{\sqrt{\omega}} \end{array}\right) \]
	associated with the right-hand side of the 
	linearized system \eqref{eq::sistem_Ulinear} has the following eigenvalues: 
	\begin{equation*}
		\lambda_{1,2} = - \frac{\beta}{2 \sqrt{\omega}} \left(  1 \pm \sqrt{1 - \frac{4 \omega}{\beta^{2}}} \right). 
	\end{equation*}
	Since $\beta>0$ and $\omega > 0$, the eigenvalues are nonzero and are either real and negative, or complex with negative real part, which proves (i). The type of the critical point depends on the sign of the discriminant $1 - 4\omega/\beta^{2}$: when $1-4\omega/\beta^2>0$, i.e., $\omega<\omega_*$, the critical point $(0,0)$ is a stable node, while when $1-4\omega/\beta^2<0$, i.e., $\omega>\omega_*$, the critical point $(0,0)$ is a stable spiral. When  $\omega=\omega_{\ast}$, $\lambda_{1,2}=-1$ with associated eigenvector $(1,-1)^{\mathrm{T}}$, and the critical point $(0,0)$ of the linearized system \eqref{eq::sistem_Ulinear} is then a stable inflected node.
\end{proof} 

\begin{corollary}\label{corr1}
	The critical point \eqref{eq::critical_U} of the nonlinear autonomous system \eqref{eq::sistem_U} is a stable node for $\omega < \omega_*$, a stable spiral for $\omega>\omega_*$, and a stable inflected node for $\omega=\omega_*$, where $\omega_{\ast}$ is as in \eqref{eq::crit}. 
\end{corollary}

\begin{proof}
	The assertion is a direct consequence of the statements (a), (b) and (c) of Theorem \ref{thm:lin} and the Hartman--Grobman theorem (cf., for example, Sec. 2.8 of \cite{perko2001}), thanks to the fact that each of the eigenvalues, $\lambda_{1,2}$, calculated in the proof of Theorem \ref{thm:lin} has a strictly negative real part: $\mathrm{Re}(\lambda_{1,2})=-\beta/(2\sqrt{\omega})<0$ for all $\beta>0$ and $\omega>0$.
\end{proof}

Corollary \ref{corr1} implies that the critical point/equilibrium height is approached monotonically in the subcritical case ($\omega < \omega_{\ast}$) and oscillatorily in the supercritical case ($\omega > \omega_{\ast}$).

While Theorem \ref{prop::spektralnaAnaliza_slip}  provides a helpful characterization of the critical point, it holds only locally, in the sense that the monotonic or oscillatory behavior asserted in the theorem can be claimed only after the trajectory of the solution has reached a sufficiently small neighborhood of the critical point, and there is no mathematically rigorous guarantee that the solution of the initial-value problem will reach that neighborhood. An attempt was made in \cite{plociniczak2018monotonicity} to prove global stability by constructing a suitable Lyapunov function for equation \eqref{eq::H_jednacina} and the corresponding basin of attraction. However, the basin of attraction determined that way did not include the initial data. 

In this work, we suggest a new Lyapunov function related to \eqref{eq::transformisana_Ujednacina}, which enables us to extend the basin of attraction and include the initial data. 

\begin{lemma}
	Let $\Omega = \left\{ (u,v) \in \mathbb{R}^2 : u \geq 0,\, v \in \mathbb{R} \right\}$; then, the function $V$ defined by
	\begin{equation}\label{eq:lyapunov}
		V(u,v) := \frac{1}{2}v^2 - u + \frac{2\sqrt{2}}{3}u^{3/2} + \frac{1}{6}, \quad \mbox{for $(u,v) \in \Omega$},
	\end{equation}
	\label{lema::LyapunovFunction}
	is a Lyapunov function for the nonlinear autonomous system \eqref{eq::sistem_U}. 
\end{lemma}
\begin{proof}
	We begin by noting that $V \in C^1(\Omega)$. 
	To confirm that $V$ is a Lyapunov function for the system \eqref{eq::sistem_U} it will be helpful to define the function
	\begin{equation}
		E(u, v) := \frac{1}{2}v^2 - u + \frac{2\sqrt{2}}{3}u^{3/2}.
		\label{eq::EnergyFunction}
	\end{equation}
	We note that at the critical point \eqref{eq::critical_U} we have
	\begin{equation*}
		E(u_{\mathrm{cr}}, v_{\mathrm{cr}}) = - \frac{1}{6}.
	\end{equation*}
	Hence, 
	\begin{equation*}
		V(u,v) =  
		\frac{1}{2}v^2 - u + \frac{2\sqrt{2}}{3}u^{3/2} + \frac{1}{6} = E(u,v) - E(u_{\mathrm{cr}}, v_{\mathrm{cr}}).
	\end{equation*}
	We are now ready to verify that the function $V$ is indeed a Lyapunov function for the system. We shall do so in three steps. 
	\begin{itemize}
		\item[(Step 1)] Clearly, $V (u_{\mathrm{cr}}, v_{\mathrm{cr}}) = V \big( \frac{1}{2},0 \big) = 0$. 
		
		\item[(Step 2)] Next, we need to show that $V(u,v) > 0$ for every $(u,v) \in \Omega$, $(u,v) \neq (u_{\mathrm{cr}}, v_{\mathrm{cr}})$. To this end, $V(u,v)$ is transformed as follows:
		\begin{equation*}
			\begin{split}
				V(u,v) & = \frac{1}{2}v^2 - u + \frac{2\sqrt{2}}{3}u\sqrt{u} + \frac{1}{6} \\
				& = \frac{1}{2}v^2 + \underbrace{ \frac{2\sqrt{2}}{3} \left( \sqrt{u} - \frac{1}{\sqrt{2}}\right)^{2} \left( \sqrt{u} + \frac{1}{2\sqrt{2}} \right).}_{> 0 \text{ for $u \in [0, \infty) \symbol{92} \{ \frac{1}{2} \} $}}
			\end{split}
		\end{equation*}    
		Therefore, the function $V(u,v)$ is strictly positive on $\Omega\setminus\{(u_{\textrm{cr}}, v_{\textrm{cr}})\}$.
		
		\item[(Step 3)]  We also need to show that $V'(u(s),v(s)):=\frac{\mathrm{d}}{\mathrm{d}s} V(u(s),v(s)) \leq 0 $, for every $(u(s),v(s)) \in \Omega$. Indeed,
		\begin{equation}
			\begin{split}
				\frac{\mathrm{d}}{\mathrm{d}s} V(u(s),v(s)) &= v(s)v'(s) - u'(s) + \frac{2\sqrt{2}}{3} \cdot \frac{3}{2} [u(s)]^{1/2} u'(s) \\
				&= v(s)(-\sqrt{2u(s)} - \frac{\beta}{\sqrt{\omega}}v(s) + 1) - v(s) + \sqrt{2}[u(s)]^{1/2} v(s) \\
				&= -\frac{\beta}{\sqrt{\omega}}[v(s)]^2,
			\end{split}
			\label{eq::V_prime}
		\end{equation}
		which is nonpositive for all  $\beta > 0$ and all $\omega>0$.   
	\end{itemize}
\end{proof} 

As $V'(u(s),v(s)) \leq 0$ for all $s \geq 0$, but $V'(u_{\mathrm{cr}}, v_{\mathrm{cr}})=V'(1/2,0)=0$ instead of $<0$, Lyapunov's stability method cannot be applied, and we have to resort to LaSalle's invariance principle to find the basin of attraction and prove the asymptotic stability of the critical point. 

\begin{theorem} \label{thm:Lyapunov}
	The unique critical point $(u_{\mathrm{cr}}, v_{\mathrm{cr}}) = \left(1/2,0 \right)$ of \eqref{eq::sistem_U} is asymptotically stable for all trajectories starting inside the set
	\begin{equation*}
		\Omega_0 = \Bigg\{ (u,v) \in \Omega : \frac{1}{2}v^2 - u + \frac{2\sqrt{2}}{3} u^{3/2} + \frac{1}{6} \leq C \Bigg\},
	\end{equation*}
	where $C$ is a constant given by
	\begin{equation*}
		C = -\frac{1}{2}\alpha^2 + \frac{2\sqrt{2}}{3}\frac{\alpha^2}{2} \sqrt{\frac{\alpha^2}{2}} + \frac{1}{6}, \quad \alpha \in \left[0,\frac{3}{2}\right]. 
	\end{equation*}
	
\end{theorem}

\begin{proof}
	Thanks to the definition \eqref{eq:lyapunov} of the Lyapunov function $V$, we have that
	\begin{equation*}
		\Omega_0 := \Bigg\{ (u,v) \in \Omega : V(u,v) \leq C \Bigg\},
	\end{equation*}
	where
	\begin{equation*}
		C := -\frac{1}{2}\alpha^2 + \frac{2\sqrt{2}}{3}\frac{\alpha^2}{2} \sqrt{\frac{\alpha^2}{2}} + \frac{1}{6}, \quad \alpha \in \left[0,\frac{3}{2}\right].
	\end{equation*}
	Because of the strict convexity of the function $V: (u,v) \in \Omega \mapsto V(u,v)$, $\Omega_0$ is the smallest level set of $V$ that is both a subset of $\Omega$ and which contains the initial state \eqref{eq::transformisani_pocetniUslovi} as well as the critical point. Next, we define
	\begin{equation*}
		\mathcal{E} := \Bigg\{ (u,v) \in \Omega_0 : V'(u,v) = 0 \Bigg\}.
	\end{equation*}
	By noting \eqref{eq::V_prime}, we deduce that all points in $\mathcal{E}$ are of the form $(u,0)$, $u \in [u_{\mathrm{min}},u_{\mathrm{max}}]$, where $u_{\mathrm{min}}$ and $u_{\mathrm{max}}$ are nonnegative solutions of the equation 
	\begin{equation*}
		- u + \frac{2\sqrt{2}}{3} u^{3/2} + \frac{1}{6} = C. 
	\end{equation*}
	By recalling the definition of $C$ a simple calculation based on factorization yields that, if $\alpha \in [0,1]$, then  
	\begin{equation}
		u_{\mathrm{min}} = \frac{\alpha^{2}}{2}, \quad 
		u_{\mathrm{max}} = \frac{9}{8} \left( \frac{1}{2} - \frac{\alpha}{3} + \frac{1}{6} \sqrt{9 + 12 \alpha - 12 \alpha^{2}} \right)^{2}. 
		\label{eq::U_bounds}
	\end{equation}
	We note that for $\alpha \in (0,1)$, $0<u_{\mathrm{min}}<u_{\mathrm{max}}<9/8$, i.e., $u_{\mathrm{min}}$ and $u_{\mathrm{max}}$ in \eqref{eq::U_bounds} satisfy the two-sided bound \eqref{eq::EnergyEstimate}. 
	For $\alpha=0$ we have that $u_{\mathrm{min}}=0$ and $u_{\mathrm{max}}=9/8$. For $\alpha=1$, $u_{\mathrm{min}} = u_{\mathrm{max}}=1/2 \in (0,9/8)$, in agreement with the fact that $u_{\mathrm{cr}}=1/2$ is the unique stationary solution. If on the other hand $\alpha \in [1,3/2]$, then
	\[u_{\mathrm{min}}= \frac{9}{8} \left( \frac{1}{2} - \frac{\alpha}{3} + \frac{1}{6} \sqrt{9 + 12 \alpha - 12 \alpha^{2}} \right)^{2}, \quad u_{\mathrm{max}}=\frac{\alpha^2}{2},\]
	and, again, $0<u_{\mathrm{min}}<u_{\mathrm{max}}<9/8$ for $\alpha \in (1,3/2)$, while for $\alpha = 3/2$ we have that $u_{\mathrm{min}}=0$ and $u_{\mathrm{max}}=9/8$.
	
	Let us suppose that $\alpha \in [0,3/2]$. Then, $0\leq u_{\mathrm{min}}\leq u_{\mathrm{max}} \leq 9/8$.
	To prove the asymptotic stability of the critical point \eqref{eq::critical_U} we have to determine the largest forward invariant subset $\mathcal{M}$ of $\mathcal{E}$, i.e., the largest set of solutions of the form $(u(s),0)$ with initial data in $\mathcal{E}$, i.e., with $(u(0),v(0)) = (u_{0},0)$, where $u_{0} \in [u_{\mathrm{min}},u_{\mathrm{max}}]$. It is obvious that $\mathcal{M}$ contains at least one point---the critical point itself. Our goal is to show that it is the only point contained in $\mathcal{M}$. To do so, we return to \eqref{eq::sistem_U} and require that $v' = 0$, since $v = 0$ for all $t \geq 0$. However, it is obvious from \eqref{eq::sistem_U} that $v' = 1 - \sqrt{2 u} \neq 0$ when $v = 0$, for any $u \neq u_{\mathrm{cr}} = 1/2$. It follows that there is no other point in $\mathcal{M}$ except the critical point $(u_{\mathrm{cr}}, v_{\mathrm{cr}})$, i.e., 
	\begin{equation*}
		\mathcal{M} = \left\{ \left( \frac{1}{2},0 \right) \right\}.
	\end{equation*}
	
	To complete the proof we still need to analyze solutions with initial points in $\partial\Omega_{0} \cap \mathcal{E} = \{ (u_{\mathrm{min}},0), (u_{\mathrm{max}},0) \}$. Our goal is to ensure that these two points are in the basin of attraction, $\Omega_0$. To that end, consider the initial-value problem for the system \eqref{eq::sistem_U} with $(u(0),v(0)) = (u_{\mathrm{min}}, 0)$. It follows from the definition of $\Omega_{0}$ that $V(u(0),v(0)) = C$. We know from Theorem \ref{th::global} that there exists a unique solution $(u(s),v(s))$ for $s \in [0,h]$ and any $h>0$. Let us assume (for contradiction) that  $V(u(h),v(h)) = \bar{C} > C$ for some $h>0$, i.e., that  $(u(h),v(h)) \notin \Omega_{0}$. Since $s \mapsto V(u(s),v(s))$ is a continuous function defined on $[0,h]$, which is differentiable on $(0,h)$, by the mean value theorem there exists a $\sigma \in (0,h)$ such that 
	\begin{equation*}
		V'(u(\sigma),v(\sigma)) = 
		\frac{V(u(h),v(h)) - V(u(0),v(0))}{h - 0}
		= \frac{\bar{C} - C}{h} > 0. 
	\end{equation*}
	This contradicts Lemma \ref{lema::LyapunovFunction}, which states that $V'(u,v) \leq 0$, $(u,v) \in \Omega$. Therefore,  $V(u(h),v(h)) \leq C$ for all $h \in [0,\infty)$, i.e.,  $(u(h),v(h)) \in \Omega_{0}$ for all $h \in  [0, \infty)$. The same argument applies to $(u(0),v(0)) = (u_{\mathrm{max}}, 0)$. 
	By applying LaSalle's invariance principle, it follows that every solution starting in $\Omega_0$ approaches $\mathcal{M}$ as $t \to \infty$.
\end{proof} 

Theorem \ref{thm:Lyapunov} facilitates the formulation of a new result regarding the solution of Washburn's equation. First, let us recall that we have proved the global existence and uniqueness of the solution of the initial-value problem \eqref{eq::transformisana_Ujednacina}, \eqref{eq::transformisani_pocetniUslovi}. Second, we have proved asymptotic stability of the critical point for the trajectories starting in the set $\Omega_{0}$ whose boundary contains our initial data. An immediate consequence of this is the following result. 

\begin{corollary}
	Let $\alpha \in [0,3/2]$. The unique solution $t \in [0,\infty) \to (u(t),v(t))$ of the system \eqref{eq::sistem_U} with initial data $(u(0),v(0)) = \left(\alpha^{2}/2, 0 \right)$ converges to the unique critical point $(u_\mathrm{cr}, v_\mathrm{cr}) = \left( 1/2, 0 \right)$ as $t \to \infty$.
\end{corollary}

In more physical terms, every solution of Washburn's equation \eqref{switala_scaled_with_imemersion1}, which starts from a small nonnegative initial height $H(0) = \alpha$ from the state of rest ($H'(0) = 0$), converges to the equilibrium height $H_0 = 1$, i.e., $H(T) \to 1$, as $T \to \infty$, regardless of the choice of the slip parameter $\beta>0$ and the dimensionless parameter $\omega>0$. This result is physically expected and was computationally observed, but lacked a rigorous mathematical proof, which is now available. 



\section{Conclusions}


We derived Washburn's capillary rise equation with a slip condition from first principles. A new scaling approach was introduced and different flow regimes were analyzed. We then proved the global existence and uniqueness of a positive solution. Furthermore, we showed that the slip parameter does not affect the qualitative behavior of the solution. Finally, a Lyapunov function was constructed, demonstrating that the system reaches a stationary solution from the given initial data. 
Future work could focus on analyzing the slip parameter, particularly investigating the effects of negative values of the slip parameter $\beta$ or its variation along the $z$-axis. A potential physical interpretation of negative slip could be related to changes in the contact angle. The construction and mathematical analysis of stable and convergent numerical methods for the approximate solution of Washburn's equation that correctly reproduce the long-time dynamics is a further potential area of future research.

\section*{Acknowledgement}

This work was supported by the Serbian Ministry of Science, Technological Development and Innovation through the Mathematical Institute of the Serbian Academy of Sciences and Arts.


\bibliographystyle{siam}
\bibliography{references}


\appendix

\section{Local forms of balance laws}

The analysis of the restrictions on the velocity and pressure fields requires the use of the local forms of mass and momentum balance laws:
\begin{equation*}
	\nabla \cdot \mathbf{v} = 0, 
	\label{eq::MassBL1}
\end{equation*}
\begin{equation*}
	\rho \left( \frac{\partial \mathbf{v}}{\partial t} + \mathbf{v} \cdot \nabla \mathbf{v} \right) = - \nabla p + \mu \nabla^{2} \mathbf{v} + \rho \mathbf{g}. 
	\label{eq::MomentumBL1}
\end{equation*}
Specifically, when the motion is studied in cylindrical coordinates, the local form of the mass balance law for an incompressible fluid reads as follows \cite{acheson1990elementary}: 
\begin{equation}
	\frac{1}{r} \frac{\partial}{\partial r} (rv_r) + \frac{1}{r} \frac{\partial v_{\varphi}}{\partial \varphi} + \frac{\partial v_z}{\partial z} = 0.
	\label{eq::MassBalanceCylindrical}
\end{equation}
The local form of the momentum balance law in cylindrical coordinates is
\begin{equation}
	\begin{split}
		\frac{\partial v_r}{\partial t} + (\mathbf{v} \cdot \nabla)v_r - \frac{v_{\varphi}^{2}}{r} &= -\frac{1}{\rho}\frac{\partial p}{\partial r} + \nu \left( \nabla^{2}v_r - \frac{v_r}{r^{2}} - \frac{2}{r^{2}} \frac{\partial v_{\varphi}}{\partial \varphi} \right), \\
		\frac{\partial v_{\varphi}}{\partial t} + (\mathbf{v} \cdot \nabla) v_{\varphi} + \frac{v_r v_{\varphi}}{r} &= -\frac{1}{\rho r} \frac{\partial p}{\partial \varphi} + \nu \left( \nabla ^{2} v_{\varphi} - \frac{v_{\varphi}}{r^2} + \frac{2}{r^{2}} \frac{\partial v_r}{\partial \varphi} \right), \\
		\frac{\partial v_z}{\partial t} + (\mathbf{v} \cdot \nabla)v_z &= - \frac{1}{\rho} \frac{\partial p}{\partial z} + \nu \nabla^{2} v_z - g.
	\end{split}
	\label{eq::MomentumBalanceCylindrical}
\end{equation}

\section{Global forms of balance laws and Washburn's equation}

Our aim is to show that Washburn's equation is a consequence of the global form of momentum balance and the assumptions on the flow field. This procedure will be split in several steps. 

\subsection{Domain}

In order to apply the balance equations in their global form, we will describe the pipe in cylindrical coordinates. The volume occupied by the fluid at time $t \geq 0$ will be taken to be the set $\mathcal{P}_{t}$ in cylindrical coordinates
\begin{equation*}
	\mathcal{P}_{t} = \{ (r, \varphi, z) \vert r \in [0, R], \varphi \in [0, 2\pi], z \in [0, h(t)] \},
\end{equation*}
where $h(t)$ is the height reached by the fluid column at time $t \geq 0$.
The boundary of $\mathcal{P}_{t}$ will be split in three parts
\begin{equation*}
	\partial \mathcal{P}_{t} = \partial \mathcal{P}_{t1} \cup \partial \mathcal{P}_{t2} \cup \partial \mathcal{P}_{t3},
\end{equation*}
so that $\partial \mathcal{P}_{t1}$ and $\partial \mathcal{P}_{t3}$ are the boundaries of the bottom and top bases of the cylinder, respectively, $\partial \mathcal{P}_{t2}$ is the boundary of the vertical side of the cylinder:
\begin{alignat*}{2}
	\partial \mathcal{P}_{t1} & = \{ (r, \varphi, z) \vert r \in [0,R], \varphi \in [0, 2\pi], z = 0\}, & \quad \mathbf{n}_{1} & = - \mathbf{e}_{z};
	\\
	\partial \mathcal{P}_{t2}  & = \{(r, \varphi, z) \vert r=R, \varphi \in [0, 2\pi], z \in [0, h(t)] \}, & \quad \mathbf{n}_{2} & = \mathbf{e}_{r};
	\\
	\partial \mathcal{P}_{t3} & = \{ (r, \varphi, z) \vert r \in [0,R], \varphi \in [0, 2\pi], z = h(t)\}, & \quad \mathbf{n}_{3} & = \mathbf{e}_{z};
\end{alignat*}
$\mathbf{n}_{1}$, $\mathbf{n}_{2}$ and $\mathbf{n}_{3}$ are the respective unit outward normal vectors. 

\subsection{Mean velocity}

As a consequence of our assumption of an axially symmetric flow field, the local mass balance law and Poiseuille velocity profile, the velocity field has the form $\mathbf{v} = v_{z}(r,t) \mathbf{e}_{z}$, where $v_{z}(r,t)$ is given, by \eqref{eq::Poiseuille}, as
\begin{equation*}
	v_{z}(r,t) = v(t) \left( 1 - \frac{r^2}{R^2} + \frac{2L}{R}  \right).
\end{equation*}
We need the mean velocity of the fluid, i.e.,
the average value of the velocity field $\mathbf{v}$ with respect to a cylindrical cross section perpendicular to the axis of the pipe:
\begin{equation*}
	\bar{v}(t) := \frac{1}{R^{2} \pi} \int_{0}^{2\pi} \int_{0}^{R} v_{z}(r,t) r \mathrm{d}r \mathrm{d}\varphi.
\end{equation*}
Using \eqref{eq::Poiseuille} and applying Fubini's Theorem we get
\begin{equation*}
	\bar{v}(t) = \frac{1}{R^{2}\pi} \int_{0}^{2\pi}\int_{0}^{R} v(t) \left( 1 - \frac{r^{2}}{R^{2}} + 2\frac{L}{R} \right)r\mathrm{d}r\mathrm{d}\varphi 
	= \frac{v(t)}{2} \left( 1 + 4 \frac{L}{R} \right). 
\end{equation*}
In subsequent computations the following form of the mean velocity is used:
\begin{equation}
	\bar{v}(t) = \frac{v(t)}{2 \beta}, \quad 
	\beta = \left( 1 + 4 \frac{L}{R} \right)^{-1}.
	\label{MeanVelocity}
\end{equation}

\subsection{Momentum balance law in global form}

We want to apply the momentum balance in global form 
\begin{equation}
	\frac{\mathrm{d}}{\mathrm{d}t} \int_{\mathcal{P}_{t}} \mathbf{K}(\mathbf{x},t) \mathrm{d}V = \int_{\partial \mathcal{P}_t}\mathbf{t(n)}\mathrm{d}S + \int_{\mathcal{P}_{t}} \rho \mathbf{b} \mathrm{d}V,
	\label{eq::MomentumBalanceGlobal}
\end{equation}
to the volume $\mathcal{P}_{t}$. We will calculate separately the left- and the right-hand side of the equation.

\subsubsection{Left-hand side}

The left-hand side expresses the rate of change of momentum of the fluid. The momentum density is given by
\begin{equation*}
	\mathbf{K}(\mathbf{x},t) = \rho(\mathbf{x},t) \mathbf{v}(\mathbf{x},t) = \rho v_{z} (r,t) \mathbf{e}_z.
\end{equation*}
Thus, we have
\begin{equation*}
	\frac{\mathrm{d}}{\mathrm{d}t} \int_{\mathcal{P}_{t}} \mathbf{K}(\mathbf{x},t) \mathrm{d}V = \frac{\mathrm{d}}{\mathrm{d}t} \int_{\mathcal{P}_{t}} \rho v_z(r,z)\mathbf{e}_z \mathrm{d}V = \rho \left[ \frac{\mathrm{d}}{\mathrm{d}t} \int_{0}^{h(t)} \int_{0}^{R} \int_{0}^{2\pi} v_z(r,t) \mathrm{d}\varphi r\mathrm{d}r\mathrm{d}z \right] \mathbf{e}_z.
\end{equation*}
Taking into account \eqref{eq::Poiseuille} and 
\begin{equation*}
	\int_{0}^{R} \int_{0}^{2\pi} v_z(r,t)\mathrm{d}\varphi r\mathrm{d}r = R^{2}\pi \bar{v}(t),
\end{equation*}
it follows that
\begin{equation}
	\frac{\mathrm{d}}{\mathrm{d}t} \int_{\mathcal{P}_{t}} \mathbf{K}(\mathbf{x},t) \mathrm{d}V = \rho R^{2}\pi \left[ \frac{\mathrm{d}}{\mathrm{d}t} \int_{0}^{h(t)} \bar{v}(t) \mathrm{d}z \right] \mathbf{e}_z 
	= \rho R^{2}\pi \frac{\mathrm{d}}{\mathrm{d}t} \left[ \bar{v}(t) h(t) \right] \mathbf{e}_z.
	\label{MomentumGlobal}
\end{equation}

\subsubsection{Right-hand side}

The right-hand side of \eqref{eq::MomentumBalanceGlobal} involves contact and body forces acting on $\mathcal{P}_{t}$. For the body forces we have 
\begin{equation*}
	\mathbf{F}_v = \int_{\mathcal{P}_{t}} \rho \mathbf{b} \mathrm{d}V = \int_{\mathcal{P}_{t}} \rho (-g\mathbf{e}_{z}) \mathrm{d}V = - \rho g \mathbf{e}_{z} \int_{0}^{h(t)} \int_{0}^{R} \int_{0}^{2\pi} r \mathrm{d}r \mathrm{d}\varphi \mathrm{d}z.
\end{equation*}
Fubini's theorem gives that
\begin{equation}
	\mathbf{F}_v = \int_{\mathcal{P}_{t}} \rho \mathbf{b} \mathrm{d}V = - \rho g R^{2}\pi h(t) \mathbf{e}_{z}.
	\label{eq::bodyForce}
\end{equation}

It is left to determine the total contact force $\mathbf{F}_s$ by integrating the traction over the boundary of the pipe. Since the boundary is divided into three regions, it follows that
\begin{equation*}
	\mathbf{F}_s = \int_{\partial \mathcal{P}_{t}} \mathbf{t}(\mathbf{n}) \mathrm{d}S = \int_{\partial \mathcal{P}_{t1}}\mathbf{t}_{1}(\mathbf{n}_{1}) \mathrm{d}S_1 + \int_{\partial \mathcal{P}_{t2}} \mathbf{t}_{2}(\mathbf{n}_{2}) \mathrm{d}S_2 + \int_{\partial \mathcal{P}_{t3}} \mathbf{t}_{3}(\mathbf{n}_{3}) \mathrm{d}S_3,
\end{equation*}
where $S_1, S_2$ and $S_3$ are surface elements. To this end it will be helpful to recall the following expression for the traction of an incompressible viscous fluid (cf. \cite{acheson1990elementary,Kuzmanovic2022}):
\begin{equation}
	\mathbf{t}(\mathbf{n}) = - p \mathbf{n} + \mu \left[ 2 (\mathbf{n} \cdot \nabla) \mathbf{v} + \mathbf{n} \times (\nabla \times \mathbf{v}) \right]. 
	\label{eq::StressAcheson}
\end{equation} 
Since in our case
\begin{equation*}
	\nabla = \mathbf{e}_{r} \frac{\partial}{\partial r} 
	+ \mathbf{e}_{\varphi} \frac{1}{r} \frac{\partial}{\partial \varphi} + \mathbf{e}_{z} \frac{\partial}{\partial z}, 
\end{equation*}
we can then determine the local tractions. For $\mathbf{t}_{1}$, $\mathbf{n}_{1} = - \mathbf{e}_{z}$ and it follows that
\begin{gather*}
	(\mathbf{n}_{1} \cdot \nabla) \mathbf{v} = (-\mathbf{e}_z \cdot \nabla) (v_z\mathbf{e}_z) = -\frac{\partial v_z}{\partial z}\mathbf{e}_z = \mathbf{0}, 
	\\ 
	\mathbf{n}_{1} \times (\nabla \times \mathbf{v}) = (-\mathbf{e}_z) \times \left( -\frac{\partial v_z}{\partial r} \mathbf{e}_{\varphi} \right) = - \frac{\partial v_z}{\partial r} \mathbf{e}_{r}.
\end{gather*}
Therefore,
\begin{equation*}
	\mathbf{t}_{1}(\mathbf{n}_{1}) = p(0) \mathbf{e}_z - \mu \frac{\partial v_z}{\partial r} \mathbf{e}_r.
\end{equation*}
Similarly, for $\mathbf{t}_{2}$, $\mathbf{n}_{2} = \mathbf{e}_{r}$, and we have that 
\begin{gather*}
	( \mathbf{n}_{2} \cdot \nabla)\mathbf{v} = (\mathbf{e}_r \cdot \nabla)(v_z\mathbf{e}_z) = \frac{\partial v_z}{\partial r}\mathbf{e}_z,
	\\
	\mathbf{n}_{2} \times (\nabla \times \mathbf{v}) = \mathbf{e}_r \times \left( -\frac{\partial v_z}{\partial r}\mathbf{e}_{\varphi} \right) = -\frac{\partial v_z}{\partial r} \mathbf{e}_z,
\end{gather*}
which implies that 
\begin{equation*}
	\mathbf{t}_{2}(\mathbf{n}_{2}) = - p (z) \mathbf{e}_r +  \mu \frac{\partial v_z}{\partial r} \mathbf{e}_z. 
\end{equation*}
Finally, since $\mathbf{t}_{3} = - \mathbf{t}_{1}$ we have that
\begin{equation*}
	\mathbf{t}_{3}(\mathbf{n}_{3}) = - p(h(t)) \mathbf{e}_z + \mu \frac{\partial v_z}{\partial r} \mathbf{e}_r. 
\end{equation*}

Using these expressions, we can now compute the net force exerted on the bases of the cylinder:
\begin{equation*}
	\begin{split}
		\int_{\partial \mathcal{P}_{t1}} \mathbf{t}_{1}(\mathbf{n}_{1}) \mathrm{d}S_1 & = \int_{0}^{R} \int_{0}^{2\pi} \mathbf{t}_{1}(\mathbf{n}_{1}) r \mathrm{d}r \mathrm{d} \varphi = R^{2} \pi p(0) \mathbf{e}_{z} - 2 \pi \mu \left[ \int_{0}^{R} r \frac{\partial v_z}{\partial r} \mathrm{d}r \right] \mathbf{e}_{z}, \\
		\int_{\partial \mathcal{P}_{t3}} \mathbf{t}_{3}(\mathbf{n}_{3}) \mathrm{d}S_3 & = \int_{0}^{R} \int_{0}^{2\pi} \mathbf{t}_{3}(\mathbf{n}_{3}) r \mathrm{d}r \mathrm{d} \varphi = - R^{2} \pi p(h(t)) \mathbf{e}_{z} + 2 \pi \mu \left[ \int_{0}^{R} r \frac{\partial v_z}{\partial r} \right]\mathrm{d}r \mathbf{e}_{z},
	\end{split}
\end{equation*}
and thus
\begin{equation*}
	\int_{\partial \mathcal{P}_{t1}} \mathbf{t}_{1}(\mathbf{n}_{1}) \mathrm{d}S_1 + \int_{\partial \mathcal{P}_{t3}} \mathbf{t}_{3}(\mathbf{n}_{3}) \mathrm{d}S_3 = R^{2}\pi \left[ p(0) - p(h(t)) \right] \mathbf{e}_z.
\end{equation*}
We now consider the net force exerted on the vertical boundary of the cylinder:
\begin{equation*}
	\int_{\partial \mathcal{P}_{t2}} \mathbf{t}_{2}(\mathbf{n}_{2}) \mathrm{d}S_2 = - \left[ \int_{0}^{2\pi} \int_{0}^{h(t)} 
	p(z) \mathbf{e}_{r} R \mathrm{d}\varphi \mathrm{d}z \right] + \mu \left[ \int_{0}^{2\pi} \int_{0}^{h(t)} \frac{\partial v_z}{\partial r} R \mathrm{d}\varphi \mathrm{d}z \right] \mathbf{e}_{z}.
\end{equation*}
Note that the unit vector $\mathbf{e}_{r}$ must be kept under the integral sign because it is not constant. To compute the first integral, we use that $\mathbf{e}_{r} = \operatorname{cos}\varphi \,\mathbf{e}_{x} + \operatorname{sin} \varphi \,\mathbf{e}_{y}$ and apply Fubini's theorem to find that
\begin{align*}
	\int_{0}^{2\pi} \int_{0}^{h(t)} & p(z) \mathbf{e}_{r} R \mathrm{d}\varphi \mathrm{d}z 
	= \int_{0}^{2\pi} \int_{0}^{h(t)} p(z) \left( \operatorname{cos} \varphi \mathbf{e}_{x} + \operatorname{sin} \varphi \mathbf{e}_{y} \right) R \mathrm{d}\varphi \mathrm{d}z 
	\\ 
	& = \left[ \int_{0}^{2\pi} \int_{0}^{h(t)} p(z) \operatorname{cos} \varphi R \mathrm{d}\varphi \mathrm{d}z \right] \mathbf{e}_{x} + \left[ \int_{0}^{2\pi} \int_{0}^{h(t)} p(z) \operatorname{sin} \varphi R \mathrm{d}\varphi \mathrm{d}z \right] \mathbf{e}_{y} = \mathbf{0}.  
\end{align*}
This result is a consequence of the axially symmetric distribution of the pressure along the side of the cylinder. In order to evaluate the second term, we will use the Poisseuille flow profile \eqref{eq::Poiseuille} and the mean velocity \eqref{MeanVelocity}; thus, 
\begin{equation*}
	2\mu \pi R \left[ \int_{0}^{h(t)}  
	\frac{\partial v_z}{\partial r} \mathrm{d}z \right] \mathbf{e}_z = - 8\mu \pi \beta \bar{v}(t) h(t) \mathbf{e}_{z}.
\end{equation*}
The last equality follows when by taking $r=R$. Finally,
\begin{equation*}
	\int_{\partial \mathcal{P}_{t2}} \mathbf{t}_{2}(\mathbf{n}_{2}) \mathrm{d}S_2 = - 8\mu \pi \beta \bar{v}(t) h(t) \mathbf{e}_z.
\end{equation*}
Thus, the total contact force is 
\begin{equation}
	\mathbf{F}_s = \int_{\mathcal{P}_{t}} \mathbf{t}(\mathbf{n}) \mathrm{d}S = -R^{2}\pi  \left[p(h(t)) - p(0) \right] \mathbf{e}_z - 8\mu \pi \beta \bar{v}(t)h(t) \mathbf{e}_z. 
	\label{eq::surfaceForce}
\end{equation}

\subsection{Washburn's equation}

First, we substitute \eqref{MomentumGlobal}, \eqref{eq::bodyForce} and \eqref{eq::surfaceForce} into the global form of the momentum balance equation \eqref{eq::MomentumBalanceGlobal}; hence, 
\begin{equation*}
	\rho R^{2}\pi \frac{\mathrm{d}}{\mathrm{d}t} \left[ \bar{v}(t) h(t) \right] \mathbf{e}_z = - \rho g R^{2}\pi h(t) \mathbf{e}_{z} -R^{2}\pi \left[p(h(t)) - p(0) \right] \mathbf{e}_z - 8\mu \pi \beta \bar{v}(t)h(t) \mathbf{e}_z.
\end{equation*}
Since all vectors appearing in this equality are parallel to $\mathbf{e}_{z}$, we shall analyze only this component. Dividing both sides by $R^{2}\pi$ and substituting $\bar{v}(t) = \frac{\mathrm{d}h(t)}{\mathrm{d}t} \equiv \dot{h}(t)$ we obtain
\begin{equation}
	\left[p(0) - p(h(t)) \right] = \rho \frac{\mathrm{d}}{\mathrm{d}t} \left[\dot{h}(t)h(t)\right] + \rho gh(t) + \frac{8\mu \beta \dot{h}(t)h(t)}{R^{2}}. 
	\label{Washburn0}
\end{equation}

To complete the derivation of Washburn's equation we need to introduce the surface tension, which is the main driving agent for capillary flow of the fluid. The coefficient $\gamma$ of surface tension is defined as
\begin{equation*}
	\gamma := \frac{F}{d},
\end{equation*}
where $F$ is the intensity of the force due to surface tension and $d$ is the length along which the force is exerted. In our case, the force due to surface tension acts along the inner circumference of the pipe, whereby
\begin{equation*}
	F = 2R \pi \gamma.
\end{equation*}
Since we are considering the fluid flow in the direction of the vertical $z$-coordinate, we need only the vertical component of the surface tension force
\begin{equation*}
	F_{vert} = F \operatorname{cos}\theta = 2R \pi \gamma \operatorname{cos} \theta,
\end{equation*}
where $\theta$ is the contact angle of the fluid with the vertical wall of the pipe. On the other hand, the pressure has only a $z$-component and the global pressure difference balances the vertical component of the surface tension; therefore, 
\begin{equation*}
	p(0) - p(h(t)) = \frac{F_{vert}}{R^{2} \pi} = \frac{2 \gamma \operatorname{cos}\theta}{R}.
\end{equation*}
Substituting the last expression into \eqref{Washburn0} leads to Washburn's equation
\begin{equation}
	\rho \frac{\mathrm{d}}{\mathrm{d}t} \left[\dot{h}(t)h(t) \right] + \rho gh(t) + \frac{8\mu \beta \dot{h}(t)h(t)}{R^{2}} 
	= \frac{2 \gamma \operatorname{cos}\theta}{R}.
	\label{Washburn}
\end{equation}

\end{document}